%% file: Space_loc.tex
\documentclass[twoside, 11pt]{article}

\usepackage[latin1]{inputenc}
\usepackage[T1]{fontenc}
\usepackage[english]{babel}
\usepackage{color}
\usepackage{tocloft}
\usepackage{verbatim}

\usepackage{tikz}
\usetikzlibrary{snakes}
\usetikzlibrary{arrows}
\usepackage{tikz-cd}

\usepackage{longtable}

\usepackage[
  hmarginratio={1:1},     % equal left and right margins
  vmarginratio={1:1},     % equal top and bottom margins
  textwidth=17cm,        % new text width
  textheight=24cm,
  heightrounded          % always useful
]{geometry}

\usepackage[noindentafter]{titlesec}
\titlespacing{\paragraph}{%
  0pt}{%              left margin
0.25\baselineskip}{% space before (vertical)
1em}%               space after (horizontal)

\titlespacing{\section}{%
0pt}{%              left margin
0.2cm}{% space before (vertical)
0em}%               space after (horizontal)

\titlespacing{\subsection}{%
0pt}{%              left margin
0.2cm}{% space before (vertical)
0em}%               space after (horizontal)

\titlespacing{\subsubsection}{%
0pt}{%              left margin
0cm}{% space before (vertical)
0em}%               space after (horizontal)

\usepackage{amsmath,amsthm,amssymb}
\usepackage{mathrsfs, mathtools, xpatch}
\usepackage{dsfont}

\allowdisplaybreaks

\usepackage[all]{xy} 
\SelectTips{cm}{11} % Makes all xy arrows match the usuadifferl LaTeX arrows 
\usepackage[hidelinks]{hyperref}

\usepackage[numbers,sort&compress]{natbib}
\makeatletter
\renewcommand{\@biblabel}[1]{[#1]\hfill}
\makeatother
\usepackage[mathscr]{euscript}
\usepackage{calrsfs}
\DeclareMathAlphabet{\pazocal}{OMS}{zplm}{m}{n}

\usepackage[shortlabels]{enumitem}

\newtheoremstyle{thm}                                                           % Name
{0.15cm}                                         % Space above
{0.15cm}                                         % Space below
{\itshape}      % Body font
{}                                      % Indent amount
{\bfseries}                             % Theorem head font
{}                                      % Punctuation after theorem head
{0.2cm}                                         % Space after theorem head
{\thmname{#1}~\thmnumber{#2}\thmnote{ (#3)}}%

\makeatletter
\xpatchcmd{\proof}{\topsep6\p@\@plus6\p@\relax}{}{}{}
\makeatother

\newtheoremstyle{rmk}                                                           % Name
{0.15cm}                                         % Space above
{0.15cm}                                         % Space below
{}      % Body font
{}                                      % Indent amount
{\bfseries}                             % Theorem head font
{}                                      % Punctuation after theorem head
{0.2cm}                                         % Space after theorem head
{}                                      % Theorem head spec (can be left empty, meaning 'normal')

\theoremstyle{thm}
\newtheorem{theorem}[equation]{Theorem}

\newtheorem{corollary}[equation]{Corollary}

\newtheorem{lemma}[equation]{Lemma}
\newtheorem{proposition}[equation]{Proposition}
\newtheorem{definition}[equation]{Definition}

\theoremstyle{rmk}
\newtheorem{example}[equation]{Example}
\newtheorem{remark}[equation]{Remark}

\numberwithin{equation}{section}

\DeclareMathAlphabet{\mathbbe}{U}{bbold}{m}{n}

\input{Space_loc_Defs.tex}

\mathtoolsset{showonlyrefs,showmanualtags}

\newlength{\Displayskip}
\begin{document}

\setlength{\abovedisplayskip}{\Displayskip}
\setlength{\belowdisplayskip}{\Displayskip}

\begin{center}
\LARGE{\textbf{An $\infty$-Categorical Localisation Functor\\for Diagrams of Simplicial Sets}}
\end{center}
\vspace{0.2cm}
\begin{center}
\large Severin Bunk
\end{center}

\vspace{0.2cm}

\begin{abstract}
\noindent
Associated to each small category $\scC$, there is a category of $\scC$-shaped diagrams of simplicial sets and an $\infty$-category of $N\scC$-shaped homotopy coherent diagrams of spaces.
We present a functor which exhibits the latter as the $\infty$-categorical localisation of the former at the objectwise weak homotopy equivalences.
This builds on a Quillen equivalence between the projective and covariant model structures associated to $\scC$ due to Heuts-Moerdijk~\cite{HM:Left_fibs_and_hocolims}, as well as Cisinski's theory of $\infty$-categorical localisations~\cite{Cisinski:HCats_HoAlg}.
We use the localisation functor to give simplified proofs that the left (resp.~right) homotopy Kan extension of diagrams of simplicial sets presents the $\infty$-categorical left (resp.~right) Kan extension of coherent diagrams of spaces.
\end{abstract}

\tableofcontents

\section{Introduction}
\label{sec:Introduction}

The use of homotopy coherent mathematics, and in particular of $\infty$-categories, is becoming increasingly common.
Where available, an often powerful approach is to combine $\infty$-categorical theory with stricter classical constructions, such as in the setting of simplicial or model categories.
Presenting $\infty$-categorical constructions by stricter models can have the upshot of combining the larger scope of coherent mathematics with the easier and more explicit computability of classical constructions.
To exploit this, it is necessary that one can pass easily between both frameworks.
The main goal of this paper is to provide a short and explicit transition from strict diagrams of simplicial sets, indexed by a small category $\scC$, to coherent diagrams of spaces, indexed by the nerve $N \scC$.

More concretely, let $\sSet$ denote the category of simplicial sets, and let $\Kan \subset \sSet$ denote the full subcategory on the Kan complexes.
We combine the Quillen equivalences, treated in~\cite{HM:Left_fibs_and_hocolims}, between the projective and covariant model structures associated to a small category $\scC$ with the model for the $\infty$-category $\scS$ of spaces and the $\infty$-categorical localisation theory from~\cite{Cisinski:HCats_HoAlg} to write down a functor
\begin{equation}
	\gamma_\scC \colon N \Fun(\scC, \Kan) \longrightarrow \scFun(N\scC, \scS)
\end{equation}
from the nerve of the category of $\scC$-shaped diagrams of Kan complexes to the $\infty$-category of $N\scC$-shaped diagrams of spaces.
We then prove:

\begin{theorem}
\label{st:intro loc thm}
The functor $\gamma_\scC$ exhibits $\scFun(N\scC, \scS)$ as the $\infty$-categorical localisation of $\Fun(\scC, \sSet)$ at the objectwise weak homotopy equivalences.
Furthermore, the functor $\gamma_\scC$ allows us to compute mapping spaces, internal homs and Kan extensions in $\scFun(N\scC, \scS)$ via the classical derived constructions in the projective model structure on $\Fun(\scC, \sSet)$.
\end{theorem}

The mere fact that $\scFun(N\scC, \scS)$ is the $\infty$-categorical localisation of $\Fun(\scC, \sSet)$ at the objectwise weak equivalences is known:
one way to see this is through a combination of Lurie's~\cite[Prop.~4.2.4.4]{Lurie:HTT}, Hinich's~\cite[Prop.~1.2.1]{Hinich:DK_localisation} and classical results by Dwyer and Kan~\cite{DK:Simplicial_localisations, DK:Calculating, DK:Function_complexes}.
However, this chain of arguments is rather indirect.
It also does not lend itself easily to being used in calculations because it passes through various models for $\infty$-categories (quasicategories, fibrant simplicial categories and marked simplicial sets).
More recently, Cisinski provided an alternative proof of the fact~\cite[Thm.~7.8.9, Cor.~7.9.9]{Cisinski:HCats_HoAlg}---entirely in the world of quasicategories---in the course of developing his theory of $\infty$-categorical localisations.
While this proof is conceptually clearer and includes the provision of a localisation functor~(combining \cite[Par.~7.3.14, 7.8.5, 7.8.7]{Cisinski:HCats_HoAlg}), the latter is still hard to compute explicitly (for instance, it still involves an $\infty$-categorical left Kan extension).

The contribution of this paper lies in providing an explicit and computable localisation functor as well as a much simplified proof of the fact that $\scFun(N\scC, \scS)$ is the $\infty$-categorical localisation of $\Fun(\scC, \sSet)$.
One way of approaching the localisation is through Lurie's relative nerve functor~\cite[Def.~3.2.5.2]{Lurie:HTT}, which is part of a Quillen equivalence between the projective model structure on $\Fun(\scC, \sSet)$ and the covariant model structure on $\sSet_{/N\scC}$; this was shown in~\cite[Prop.~3.2.5.18]{Lurie:HTT} and a simplified proof was given in~\cite{HM:Left_fibs_and_hocolims}.
One could then invoke abstract arguments or further constructions which show that either of the above model categories present the $\infty$-category $\scFun(N\scC, \scS)$ (this is proved, for instance, in the aforementioned~\cite[Prop.~4.2.4.4]{Lurie:HTT} and~\cite[Thm.~7.8.9, Cor.~7.9.9]{Cisinski:HCats_HoAlg}).

Here, we circumvent this second step:
instead, we enhance the relative nerve functor to write down the functor $\gamma_\scC$, which takes its values directly in the $\infty$-category $\scFun(N\scC, \scS)$, and prove that this already exhibits the $\infty$-categorical localisation (like~\cite{HM:Left_fibs_and_hocolims}, we avoid the homotopy coherent nerve~\cite{Lurie:HTT, Stevenson:Cov_MoStrs_spl_loc} in doing so).
At the same time, we obtain a framework which unifies three perspectives on $\scS$-valued functors:
they can be presented as strict $\sSet$-valued diagrams, or as left fibrations over $N\scC$, or they can be viewed directly as objects in $\scFun(N\scC,\scS)$.
Furthermore, we make manifest the full homotopical meaning of the rectification of simplicial diagrams $F \colon \scC \to \Kan$ from~\cite{HM:Left_fibs_and_hocolims}:
it is really the left fibration presented by the coherent diagram of spaces $\gamma_\scC F \colon N\scC \to \scS$ described by $F$.

In the case where $\scC = *$ is the final category, we obtain an $\infty$-categorical localisation functor $\gamma \colon N\Kan \to \scS$, which exhibits the localisation at the weak homotopy equivalences.
We show an explicit instance of the fact that localisation commutes with taking functor categories; that is, we show that there is a commutative diagram of $\infty$-categories and functors (Theorem~\ref{st:gamma_C = gamma o N o dashv})
\begin{equation}
\begin{tikzcd}[column sep=1cm, row sep=1cm]
	N\Fun(\scC, \Kan) \ar[d, "\gamma_\scC"'] \ar[r, "\cong"]
	& \scFun(N\scC, N\Kan) \ar[d, "\gamma_*"]
	\\
	\scFun(N\scC, \scS) \ar[r, equal]
	& \scFun(N\scC, \scS)
\end{tikzcd}
\end{equation}

Finally, we remark that shortly before the completion of the present paper, the preprint~\cite{Sharma:coCart_fibs_and_hocolims} appeared on the ArXiv.
There, the Quillen equivalences from~\cite{HM:Left_fibs_and_hocolims} are lifted from the projective and covariant model structures based on the Kan-Quillen model structure to those based on the Joyal model structure on $\sSet$.
Employing Sharma's results it should be possible to also prove versions of the results in the present paper for diagrams of $\infty$-categories in place of diagrams of $\infty$-groupoids.

\paragraph*{Outline.}
In Section~\ref{sec:r_C! -| r_C^*} we recall the rectification and homotopy colimit functors, as well as various Quillen equivalences from~\cite{HM:Left_fibs_and_hocolims}.
We provide alternative presentations of the rectification functor and analyse its properties.
In Section~\ref{sec:from strict to coherent diagrams} we enhance the rectification functor to a functor $\gamma_\scC$ from strict diagrams of Kan complexes to coherent diagrams of spaces and investigate its interplay with left fibrations.
In Section~\ref{sec:gamma} we prove our main theorem, namely that $\gamma_\scC$ implements the $\infty$-categorical localisation of the category of strict Kan-complex-valued diagrams at the objectwise weak equivalences (Theorem~\ref{st:localisation theorem}).
We use this result and the techniques developed to prove several nice computational properties of $\gamma_\scC$ in Section~\ref{sec:gamma_C in computations}: it allows us to compute mapping spaces, exponentials and Kan extensions at the level of strict diagrams.
Finally, we recall some of Cisinski's $\infty$-categorical localisation theory in Appendix~\ref{sec:infty-localisation}.

\paragraph*{Notation}

We fix a nested pair $\rmU \in \rmV$ of Grothendieck universes such that the standard simplex category $\bbDelta$ is $\rmU$-small.
We will not distinguish notationally between a morphism $u \colon [k] \to [n]$ in $\bbDelta$ and the corresponding morphism $u \colon \Delta^k \to \Delta^n$ of simplicial sets under the Yoneda embedding of $\bbDelta$.

\vspace{-0.4cm}
\noindent
\begin{longtable}{p{3cm}@{\hspace{-.5cm}}p{14cm}}
$\rmh \scD$ & homotopy category of an $\infty$-category $\scD$
\\
$\Ho (\scM)$ & homotopy category of a model category $\scM$
\\
$\ul{\scM}(-,-)$ & simplicially enriched hom space in a simplicial category $\scM$
\\
$\scS$ & $\infty$-category of $\rmU$-small spaces/$\infty$-groupoids
\\
\smash{$\widehat{\scS}$} & $\infty$-category of $\rmV$-small spaces/$\infty$-groupoids
\\
$\sSet$ & Kan-Quillen model category of $\rmU$-small simplicial sets
\\
$\rmV\sSet$ & Kan-Quillen model category of $\rmV$-small simplicial sets
\\
$\Kan$ & full subcategory of $\sSet$ on the $\rmU$-small Kan complexes
\\
$\Fun(\scC, \scD)$ & 1-category of (1-categorical) functors $\scC \to \scD$, for 1-categories $\scC$ and $\scD$
\\
$\scFun(\scA, \scB)$ & $\infty$-category of functors from $\scA \in \sSet$ to an $\infty$-category $\scB$
\\
$\scC[W^{-1}]$ & 1-categorical/Gabriel-Zisman localisation of a 1-category $\scC$ at a set $W$ of morphisms
\\
$L^B_S \scM$ & left Bousfield localisation of a model category $\scM$ at a set $S$ of morphisms
\\
$L_W \scA$ & $\infty$-categorical localisation of an $\infty$-category $\scA$ at a set of morphisms $W$
\end{longtable}
\vspace{-0.8cm}

\paragraph*{Acknowledgements}

The author would like to thank Walker Stern for valuable feedback on a preliminary draft of this paper.
This research was funded by the Deutsche Forschungsgemeinschaft (DFG, German Research Foundation) under project number 468806966.

\section{The rectification functor}
\label{sec:r_C! -| r_C^*}

In this section we start by reviewing results from~\cite{HM:Left_fibs_and_hocolims}.
We develop different presentations of the \textit{rectification functor} $r_\scC^*$ from~\cite{HM:Left_fibs_and_hocolims} and analyse its various properties.

\subsection{Various model structures on simplicial sets and diagrams}

We briefly recall the main model structures we will use on categories of diagrams of simplicial sets.

\begin{theorem}
\label{st:covariant MoStr}
Let $A$ be a simplicial set.
There exists a unique model structure on the slice category $\sSet_{/A}$ whose cofibrations are the monomorphisms and whose fibrant objects are the left fibrations with codomain $A$~\cite[Ch.~8]{Joyal:QuasiCats_and_applications}.
In this model structure, a morphism between fibrant objects is a fibration precisely if it is a left fibration, and it is a weak equivalence if and only if it is a fibrewise weak equivalence~\cite[Thm.~4.4.14]{Cisinski:HCats_HoAlg}.
Furthermore, this model structure is simplicial~\cite[Prop.~2.1.4.8]{Lurie:HTT}.
\end{theorem}

\begin{definition}
The model structure in Theorem~\ref{st:covariant MoStr} is called the \emph{covariant model structure on $\sSet_{/A}$.}
\end{definition}

The following results are standard; see, for instance, the textbook~\cite{Hirschhorn:MoCats}.

\begin{theorem}
\label{st:proj/inj MoStr on sSet^C}
Let $\scC$ be a small category.
\begin{myenumerate}
\item There exists a simplicial model structure on $\Fun(\scC, \sSet)$ whose weak equivalences (resp.~fibrations) are the objectwise weak equivalences (resp.~fibrations).

\item There exists a simplicial model structure on $\Fun(\scC, \sSet)$ whose weak equivalences (resp.~cofibrations) are the objectwise weak equivalences (resp.~cofibrations).
\end{myenumerate}
\end{theorem}

\begin{definition}
The model structures in Theorem~\ref{st:proj/inj MoStr on sSet^C} are called the \emph{projective} and \emph{injective} model structure on $\Fun(\scC, \sSet)$, respectively.
\end{definition}

\subsection{The rectification functor and its properties}
\label{sec:r^*}

One of the main tools in this paper is what is called the \textit{rectification functor} in~\cite{HM:Left_fibs_and_hocolims}, or the \textit{relative nerve} in~\cite{Lurie:HTT}.
Given a $\rmU$-small (1-)category $\scC$, it is a functor $r_\scC^* \colon \Fun(\scC, \sSet) \to \sSet_{/N\scC}$, which translates diagrams $\scC \to \sSet$ of simplicial sets on a small category $\scC$ into simplicial sets over the nerve $N\scC$.
In this section, we recall the rectification functor following~\cite{HM:Left_fibs_and_hocolims} and develop alternative presentations of this functor.
These will be useful in proofs in later sections.
We investigate several properties of this functor, such as its behaviour under changes of the indexing category $\scC$.
This is crucial for combining the rectification functor with Cisinski's model for the $\infty$-category $\scS$ of spaces~\cite[Sec.~5.2]{Cisinski:HCats_HoAlg} later on.

Let $\scC$ be a $\rmU$-small (1-)category.
Let $\alpha \colon [n] \to \scC$ be an $n$-simplex in $N\scC$, and let $\tau_{[n]} \colon \bbDelta_{/[n]} \to [n]$ be the last-vertex functor.
Recall that this sends an object $\varphi \colon [k] \to [n]$ to the value $\varphi(k) \in [n]$ and a morphism
\begin{equation}
\begin{tikzcd}[column sep=0.4cm]
	{[k]} \ar[rr, "v"] \ar[dr, "u"']
	& & {[k']} \ar[dl, "u'"]
	\\
	& {[n]} &
\end{tikzcd}
\end{equation}
to the unique morphism $u(k) \to u'(k')$ in $[n]$.
We let $\pr \colon \bbDelta_{/[n]} \to \bbDelta$ denote the canonical projection functor and $\Delta^\cdot \colon \bbDelta \to \sSet$, $[k] \mapsto \Delta^k$ the Yoneda embedding of $\bbDelta$.

Given $\alpha \colon [n] \to \scC$ and a functor $F \colon \scC \to \sSet$, define the set of natural transformations
\begin{equation}
	(r_\scC^*F)_{n, \alpha} = \Fun(\bbDelta_{/[n]}, \sSet) (\Delta^\cdot \circ \pr,\, F \circ \alpha \circ \tau_{[n]})\,.
\end{equation}
This assembles into a functor
\begin{equation}
\label{eq:r_C^* explicit}
	r_\scC^* \colon \Fun(\scC,\sSet) \longrightarrow \sSet_{/N\scC}\,,
	\qquad
	(r_\scC^*F)_n = \coprod_{\alpha \in N_n \scC} (r_\scC^*F)_{n, \alpha}\,.
\end{equation}

Explicitly, an $n$-simplex of $r_\scC^*F$ consists of pairs $(\alpha, x)$ of an $n$-simplex $\alpha \in N \scC$ together with a family $x = (x_u)_u$ as follows:
the family is indexed by all maps $u \colon [k] \to [n]$, and $x_u$ is an element $x_u \in F(\alpha_{u(k)})_k$, or equivalently a map $\Delta^k \to F(\alpha_{u(k)})$.
These data satisfy the property that, for each factorisation
\begin{equation}
\begin{tikzcd}[column sep=0.4cm]
	{[k]} \ar[rr, "v"] \ar[dr, "u"'] & & {[k']} \ar[ld, "u'"]
	\\
	& {[n]} &
\end{tikzcd}
\end{equation}
we have a commutative diagram
\begin{equation}
\label{eq:compatibility in r_C^*F}
\begin{tikzcd}
	\Delta^k \ar[r, "x_u"] \ar[d, "v"']
	& F(\alpha_{u(k)}) \ar[d, "F(f_v)"]
	\\
	\Delta^{k'} \ar[r, "x_{u'}"']
	& F(\alpha_{u'(k')})
\end{tikzcd}
\end{equation}
where $f_v \colon \alpha_{u(k)} \to \alpha_{u'(k')}$ is the image under $\alpha \colon [n] \to \scC$ of the unique morphism $u(k) \to u'(k')$ in $[n]$.
The simplicial structure of $r_\scC^*F$ is induced by composition by maps $w \colon [n'] \to [n]$.
Explicitly, $w$ sends an $n$-simplex $(\alpha, x) \in (r_\scC^*F)_n$ to the $n'$-simplex $(w^*\alpha, w^*x)$, with $w^* \alpha = \alpha \circ w \colon [n'] \to \scC$ and
\begin{equation}
	(w^*x)_u = x_{w \circ u} \in F \big( (w^*\alpha)_{u(k)} \big)_k
	= F (\alpha_{w \circ u(k)})_k
\end{equation}
for $u \colon [k] \to [n']$.
The map $r_\scC^*F \to N\scC$ is given by sending a pair $(\alpha, x)$ to the $n$-simplex $\alpha \in N_n \scC$.
Given a morphism $f \colon F \to G$ in $\Fun(\scC, \sSet)$, the image of $f$ under $r_\scC^*$ acts by sending a pair $(\alpha, x)$ to the pair $(\alpha, f(x))$, where $f$ is applied to each element $x_u$ in the family $x = (x_u)_u$.

\begin{definition}
\label{def:rectification functor}
Let $\scC$ be a $\rmU$-small category.
The functor $r_\scC^* \colon \Fun(\scC, \sSet) \longrightarrow \sSet_{/N\scC}$ is called the \emph{rectification functor for $\scC$}.
\end{definition}

\begin{remark}
This agrees with the functor $r^*$ from~\cite[p.~9]{HM:Left_fibs_and_hocolims} (we have chosen to add the subscript $\scC$ to the notation because will need to vary the indexing category).
As Heuts and Moerdijk point out already, this also coincides with Lurie's \textit{relative nerve functor}~\cite[Def.~3.2.5.2]{Lurie:HTT}.
\qen
\end{remark}

In the following we present further equivalent ways of describing the rectification functor $r_\scC^*$.
These require less data and are therefore more convenient to work with in certain circumstances.

\begin{definition}
For $n \in \NN_0$, let $\bbSigma^n$ denote the category whose objects are pairs $[i,j]$, $i \leq j \in [n]$ (we follow~\cite{Haugseng:Iterated_spans} for the notation).
We identify a pair $[i,j]$ with the interval
\begin{equation}
	[i,j] = \{l \in [n]\, | \, i \leq l \leq j \} \subset [n]\,.
\end{equation}
There is a unique morphism $[i,j] \to [i', j']$ in $\bbSigma^n$ whenever there is an inclusion of intervals $[i,j] \subset [i',j']$, or, equivalently, whenever $i' \leq i \leq j \leq j'$.
\end{definition}

\begin{remark}
The categories $\bbSigma^n$, with $n$ ranging over the non-negative integers, are the foundation for Segal objects of spans in $\infty$-categories.
We refer to~\cite{Haugseng:Iterated_spans} for more on this.
Below, we rephrase the simplicial set $r_\scC^* F$ as encoding certain spans of sections of the functor $F$.
\qen
\end{remark}

Each $n$-simplex $\alpha \in N_n \scC$ determines a functor
\begin{equation}
\label{eq:Sigma alpha}
	\Sigma \alpha \colon \bbSigma^n \to \bbDelta^\opp \times \scC\
\end{equation}
as follows:
$\Sigma \alpha$ sends an object $[i,j] \in \bbSigma^n$ to the pair $([i], \alpha_j)$.
It sends an inclusion $[i,j] \subset [i',j']$ to the pair $(v, \alpha_{j,j'})$, where $v \colon [i'] \to [i]$ is defined by $v(l) = l$, for $l = 0, \ldots, i'$, and the morphism $\alpha_{j,j'}$ in $\scC$ is the image under $\alpha$ of the unique morphism $j \to j'$ in $[n]$.
The tensor-hom adjunction in the category of $\rmU$-small categories provides a canonical identification
\begin{equation}
	(-)^\dashv \colon \Fun(\scC, \sSet) \longrightarrow \Fun(\bbDelta^\opp \times \scC, \Set)\,.
\end{equation}
Given a functor $F \colon \scC \to \sSet$ we have $F^\dashv([i],c) = F(c)_i$, for each $i \in \bbDelta$ and $c \in \scC$.

\begin{lemma}
\label{st:r_C^*F|n,a using bbSigma^n}
Let $F \colon \scC \to \sSet$ be a functor.
The simplices $(\alpha, x)$ of $r_\scC^*F$ are in canonical bijection with pairs $(\alpha, y)$, where $y$ is a section of $(\Sigma \alpha)^* F^\dashv \colon \bbSigma^n \to \Set$, or, equivalently, with elements of $\lim((\Sigma \alpha)^* F^\dashv \colon \bbSigma^n \to \Set)$.
\end{lemma}

\begin{proof}
The map $(\alpha,x) \mapsto (\alpha,y)$ simply forgets part of the data; it remains to describe an inverse map to this process.
Let $u \colon [k] \to [n]$ be an object in $\bbDelta_{/[n]}$.
The morphism $u$ has a unique factorisation
\begin{equation}
\begin{tikzcd}[column sep=1.25cm, row sep=1cm]
	{[k]} \ar[r, "v"] \ar[dr, "u"' description]
	& {[k']} \ar[r, "v'"] \ar[d, "u'" description]
	& {[k'']} \ar[dl, "u''" description]
	\\
	& {[n]} &
\end{tikzcd}
\end{equation}
The pair $(u',v)$ forms the image factorisation of $u$; that is, we may identify $u'$ with the inclusion $\im(u) \subset [n]$ (where we view $\im(u) \cong [k']$ as a totally ordered set with the order induced by this inclusion).
The map $v'$ can then be identified with the inclusion $\im(u) \hookrightarrow \overline{\im(u)} = [u(0), u(k)] \subset [n]$,
where for $i \leq j \in [n]$ we let $[i,j] \subset [n]$ denote the `interval from $i$ to $j$ in $[n]$', i.e.~the totally ordered subset of $[n]$ consisting of all $l \in [n]$ with $i \leq l \leq j$.
Note that $\overline{\im(u)} \hookrightarrow [n]$ is the smallest interval in $[n]$ through which $u$ factors.
In this way, each order-preserving map $u \colon [k] \to [n]$ factors uniquely through a minimal interval inclusion $\overline{\im(u)} \hookrightarrow [n]$.

Given a second object $v \colon [l] \to [n]$ in $\bbDelta_{/[n]}$ and a morphism $w \colon u \to v$ in $\bbDelta_{/[n]}$, we obtain an induced commutative diagram (using that $v \circ w = u$)
\begin{equation}
\begin{tikzcd}
	{[k]} \ar[r] \ar[d, "w"']
	& \im(u) \ar[r, hookrightarrow] \ar[d, hookrightarrow]
	& \overline{\im(u)} \ar[d, hookrightarrow]
	\\
	{[l]} \ar[r]
	& \im(v) \ar[r, hookrightarrow]
	& \overline{\im(v)}
\end{tikzcd}
\end{equation}
in $\bbDelta_{/[n]}$.
The right-hand vertical morphism is an inclusion of intervals in $[n]$.
By these arguments and the compatibility condition~\eqref{eq:compatibility in r_C^*F} on the family $x$, we see that the data $(\alpha, x)$ are entirely determined by the morphisms given by inclusions of intervals in $[n]$, i.e.~the morphisms in $\bbSigma^n$.
Explicitly, for any morphism $u \colon [k] \to [n]$, the compatibility condition~\eqref{eq:compatibility in r_C^*F} applied to the factorisation \smash{$[k] \overset{v}{\to} \overline{\im(u)} \overset{\iota}{\to} [n]$} induces a commutative diagram
\begin{equation}
\begin{tikzcd}
	\Delta^k \ar[r, "x_u"] \ar[d]
	& F(\alpha_{u(k)}) \ar[d, equal]
	\\
	\Delta^{\overline{\im(u)}} \ar[r, "x_\iota"']
	& F(\alpha_{u(k)})
\end{tikzcd}
\end{equation}
implying that $x_u$ is entirely determined by $x_\iota$, i.e.~by the value of the family $x$ on the interval inclusion $\overline{\im(u)} \hookrightarrow [n]$.
\end{proof}

\begin{definition}
Let $\bbLambda^n \subset \bbSigma^n$ denote the full subcategory on those objects $[i,j]$ where $j-i \in \{0,1\}$.
\end{definition}

We record the following observation:

\begin{lemma}
\label{st:bbLambda to bbSigma is cofinal}
For each $n \in \NN_0$, the inclusion $\iota^n \colon \bbLambda^n \hookrightarrow \bbSigma^n$ is homotopy cofinal; that is, its nerve $N \iota^n \colon N \bbLambda^n \hookrightarrow N \bbSigma^n$ is a cofinal morphism in $\sSet$.
\end{lemma}

\begin{proof}
Given an object $[i,j] \in \bbSigma^n$, the comma category $\iota^n_{/[i,j]}$ is equivalent to $\bbLambda^{j-i}$, whose nerve is a weakly contractible simplicial set.
\end{proof}

The following is already in~\cite{HM:Left_fibs_and_hocolims} (see before Lemma~4.1); we include it here for completeness.

\begin{lemma}
\label{st:r_c^* via bbLambda}
Let $F \colon \scC \to \sSet$ be a functor, and let $F^\dashv \colon \bbDelta^\opp \times \scC \to \Set$ denote the associated functor defined by $F^\dashv([i],c) = F(c)_i$.
The elements $(\alpha, x)$ of $r_\scC^*F$ are in canonical bijection with pairs $(\alpha, z)$, where $z$ is a section of $\iota^{n*} (\Sigma \alpha)^* F^\dashv \colon \bbLambda^n \to \Set$, or, equivalently, with elements of $\lim(\iota^{n*} (\Sigma \alpha)^* F^\dashv \colon \bbLambda^n \to \Set)$.
\end{lemma}

\begin{proof}
This follows from Lemma~\ref{st:r_C^*F|n,a using bbSigma^n} together with Lemma~\ref{st:bbLambda to bbSigma is cofinal} and the fact that homotopy cofinal functors are in particular cofinal.
\end{proof}

\begin{remark}
The upshot of using $\bbSigma^n$ over $\bbLambda^n$ is that the former assemble into a (strict) functor $\bbSigma \colon \bbDelta \to \Cat$.
The categories $\bbLambda^n$ do \textit{not} produce a functor $\bbDelta \to \Cat$, but rather we only obtain a functor $\bbDelta_{int} \to \Cat$, where $\bbDelta_{int} \subset \bbDelta$ is the wide subcategory with only the \textit{inert} morphisms; see~\cite[Def.~5.1, Def.~5.2]{Haugseng:Iterated_spans} for further details.
\qen
\end{remark}

The categories $\bbSigma^{n,\opp}$ assemble into a \textit{strict} functor
\begin{equation}
	\bbSigma^\opp \colon \bbDelta \longrightarrow \Cat\,,
	\quad
	[n] \longmapsto \bbSigma^{n,\opp}\,.
\end{equation}
It sends a morphism $u \colon [n] \to [m]$ to the functor
\begin{equation}
	\bbSigma^\opp(u) \colon \bbSigma^{n,\opp} \to \bbSigma^{m,\opp}\,,
	\quad
	[i,j] \longmapsto \big[ u(i), u(j) \big]\,.
\end{equation}
The Grothendieck construction of $\bbSigma^\opp$ is a functor $\textint \bbSigma^\opp \longrightarrow \bbDelta$.
We form the (strict) pullback of categories
\begin{equation}
\label{eq:diag for alternative pres of r_C^*}
\begin{tikzcd}[column sep=2cm, row sep=1cm]
	\Big( \bbDelta_{/N\scC} \underset{\bbDelta}{\times} \textint \bbSigma^\opp \Big)^\opp
	\cong \Big( \textint (\pr^* \bbSigma^\opp) \Big)^\opp
	\ar[r] \ar[d, "\pi"']
	& \big( \textint \bbSigma^\opp \big)^\opp \ar[d]
	&
	\\
	\big( \bbDelta_{/N\scC} \big)^\opp \ar[r, "\pr"']
	& \bbDelta^\opp
\end{tikzcd}
\end{equation}
Consider the functor
\begin{equation}
\label{eq:def Sigma}
	\Sigma^\opp \colon \bbDelta_{/N\scC} \underset{\bbDelta}{\times} \textint \bbSigma^\opp
	\longrightarrow \bbDelta \times \scC^\opp
\end{equation}
defined as follows:
given an object $(\alpha \colon [n] \to \scC, [i,j] \in \bbSigma^n)$, we set $\Sigma^\opp(\alpha,[i,j]) = ([i], \alpha_j)$.
A morphism $(\alpha \colon [n] \to \scC, [i,j]) \longrightarrow (\beta \colon [k] \to \scC, [i',j'])$ in $\bbDelta_{/N\scC} \times_\bbDelta \textint \bbSigma^\opp$ is given by a morphism $v \colon [n] \to [k]$ such that $\beta \circ v = \alpha$ and $[i',j'] \subset [v(i), v(j)]$, i.e.~$v(i) \leq i' \leq j' \leq v(j)$.
The functor $\Sigma^\opp$ sends this to the pair $(v_{|[0,i]}, \beta_{j', v(j)})$ of the following data:
\begin{myenumerate}
\item $v_{|[0,i]} \colon [i] \to [i']$ is the restriction of the morphism $v \colon [n] \to [k]$ to the interval $[i] \cong [0,i] \subset [n]$; by the condition that $[i',j'] \subset [v(i), v(j)]$ we have that $v(i) \leq i'$ and this restriction factors through $[i'] \cong [0,i'] \subset [k]$.

\item The morphism $\beta_{j', v(j)}$ is the unique morphism $\beta_{j'} \to \beta_{v(j)} = \alpha_j$ induced by the functor $\beta \colon [k] \to \scC$; again, this relies on the condition that $[i',j'] \subset [v(i), v(j)]$.
\end{myenumerate}

\begin{lemma}
\label{st:Ran for strict Grothendieck constructions}
Let $G \colon \scD \to \Cat$ be a \emph{strict} functor, and let $\pi \colon (\textint G)^\opp \to \scD^\opp$ be the opposite of the Grothendieck construction of $G$.
\begin{myenumerate}
\item The fibres $\pi^{-1}(d)$ coincide with the values $G(d)^\opp$; in particular, they depend strictly functorially on $d \in \scD$.

\item Let $\scE$ be a complete category.
For any functor $S \colon (\textint G)^\opp \to \scE$, the right Kan extension $\pi_* S$ is isomorphic in $\Fun(\scD^\opp, \scE)$ to the functor
\begin{equation}
	\rmR_\pi S \colon \scD^\opp \to \scE\,,
	\quad
	d \longmapsto \lim \big( G(d)^\opp \hookrightarrow (\textint G)^\opp \overset{S}{\longrightarrow} \scE \big)\,.
\end{equation}
\end{myenumerate}
\end{lemma}

The fact that $\rmR_\pi S$ is functorial makes use of the assumption that $G \colon \scD \to \Cat$ is strict.

\begin{proof}
Claim~(1) is an immediate property of the Grothendieck construction:
An object in $(\textint G)^\opp$ consists of a pair $(d, x)$, where $d \in \scD$ and $x \in G(d)$.
A morphism $(d, x) \to (d', x')$ in $(\textint G)^\opp$ is a pair $(f, g)$, where $f \in \scD(d',d)$ and $g \in G(d)(Gf(x'),x)$.
The claim then follows by the assumption that $G$ is strict; in particular, the value of $G$ on the identity morphism of $d$ is the identity functor on $G(d)$.

The Grothendieck construction produces a functor $\textint G \to \scD$ whose nerve is a proper morphism of $\infty$-categories~\cite[Prop.~7.3.2]{Cisinski:HCats_HoAlg}; the nerve of $\pi$ is therefore a smooth morphism.
In particular, it follows that, for each $d \in \scD$, the nerve of the inclusion $G(d)^\opp = \pi^{-1}(d) \hookrightarrow d/\pi$ is a cofinal morphism in $\sSet$~\cite[Thm.~4.4.36]{Cisinski:HCats_HoAlg}.
Equivalently, the functor $G(d)^\opp = \pi^{-1}(d) \hookrightarrow d/\pi$ is homotopy cofinal, and in particular a cofinal functor of categories.
This induces the natural isomorphism $\pi_*S \cong \rmR_\pi S$.
\end{proof}

Let $\scD$ be a small category, and let $A \colon \scD^\opp \to \Set$ be a presheaf on $\scD$.
We recall the standard equivalence (see, for instance,~\cite[Rmk.~1.1.14]{Cisinski:HCats_HoAlg})
\begin{equation}
	\Psi \colon \Fun \big( (\scD_{/A})^\opp, \Set \big)
	\longrightarrow \Fun(\scD^\opp, \Set)_{/A}\,.
\end{equation}
Explicitly, given a presheaf $F$ on $\scD_{/A}$ and $d \in \scD$, we have
\begin{equation}
\label{eq:PSh(D/A) --> PSh(D)/A}
	\Psi F(d) = \bigg( \coprod_{\phi \in A(d)} F(\phi)
	\longrightarrow \coprod_{\phi \in A(d)} * \cong A(d) \bigg)\,.
\end{equation}
Let $h \colon \scD \to \Fun(\scD^\opp, \Set)$ denote the Yoneda embedding of $\scD$.
The inverse functor of $\Psi$ sends an object $\psi \colon X \to A$ in $\Fun(\scD^\opp, \Set)_{/A}$ to the presheaf on $\scD_{/A}$ whose value on $\phi \colon h_d \to A$ is the set of factorisations \smash{$\phi = (h_d \to X \overset{\psi}{\to} A)$} of $\phi$ through $\psi$.

In the particular case where $\scD = \bbDelta$ and $A = N \scC$, this provides an equivalence
\begin{equation}
	\Psi \colon\Fun \big( (\bbDelta_{/N\scC})^\opp, \Set \big)
	\longrightarrow \Fun(\bbDelta^\opp, \Set)_{/N\scC} = \sSet_{/N\scC}\,.
\end{equation}
With $\Sigma$ as in~\eqref{eq:def Sigma}, we introduce the the composite $s_\scC^* \coloneqq \Psi \circ \pi_* \circ \Sigma^* \circ (-)^\dashv$, i.e.
\begin{equation}
\begin{tikzcd}[row sep=0.cm]
	s_\scC^* \colon \Fun(\scC, \sSet) \ar[r, "{(-)^\dashv}"]
	& \Fun(\bbDelta^\opp \times \scC, \Set) \ar[r, "\Sigma^*"]
	& \Fun \big( (\bbDelta_{/N\scC})^\opp \times_{\bbDelta^\opp} (\textint \bbSigma^\opp)^\opp,\, \Set \big) \ar[r, "\pi_*"]
	& \cdots
	\\
	\cdots \ \Fun \big( (\bbDelta_{/N\scC})^\opp, \Set \big) \ar[r, "\Psi"]
	& \sSet_{/N\scC}\,.
	& &
\end{tikzcd}
\end{equation}
Analogously, we introduce the composite
\begin{equation}
	t_\scC^* \coloneqq \Psi \circ \rmR_\pi \circ \Sigma^* \circ (-)^\dashv
	\colon \Fun(\scC, \sSet) \longrightarrow \sSet_{/N\scC}\,.
\end{equation}
We thus obtain three functors
\begin{equation}
	r_\scC^*\,, s_\scC^*\,,
	t_\scC^* \colon \Fun(\scC, \sSet) \longrightarrow \sSet_{/N\scC}\,.
\end{equation}

\begin{proposition}
\label{st:alternative presentation of r_C^*}
There are canonical isomorphisms
\begin{equation}
	r_\scC^*(F)
	\cong s_\scC^* (F)
	\cong t_\scC^*(F)\,,
\end{equation}
in $\sSet_{/N\scC}$, natural in $F \in \Fun(\scC, \sSet)$.
\end{proposition}

\begin{proof}
This is now a combination of Lemma~\ref{st:r_C^*F|n,a using bbSigma^n}, equation~\eqref{eq:PSh(D/A) --> PSh(D)/A} and Lemma~\ref{st:Ran for strict Grothendieck constructions}.
Note that the coproduct in our original definition~\eqref{eq:r_C^* explicit} corresponds here to the coproduct in the construction of the functor $\Psi$~\eqref{eq:PSh(D/A) --> PSh(D)/A}.
\end{proof}

This allows us to show the following properties of $r_\scC^*$, which are important in applications:

\begin{corollary}
The functor $r_\scC^* \colon \Fun(\scC, \sSet) \longrightarrow \sSet_{/N\scC}$ preserves filtered colimits.
\end{corollary}

\begin{proof}
We use the natural isomorphism $r_\scC^*F \cong \Psi \circ \rmR_\pi \circ \Sigma^* (F^\dashv)$.
Since colimits in functor categories are computed objectwise, $(-)^\dashv$ and $\Sigma^*$ preserve all colimits.
The value of $\rmR_\pi \circ \Sigma^* (F^\dashv)$ at an object $\alpha \colon [n] \to \scC$ of $\bbDelta_{/[n]}$ is computed as a \textit{finite} limit over $\bbSigma^n$.
Thus, $\rmR_\pi$ commutes with filtered colimits.
Finally, the equivalence $\Psi$ preserves all colimits.
\end{proof}

\begin{lemma}
The functor $r_\scC^* \colon \Fun(\scC, \sSet) \longrightarrow \sSet_{/N\scC}$ sends injective cofibrations (i.e.~objectwise monomorphisms) to cofibrations (i.e.~monomorphisms over $N\scC$).
\end{lemma}

\begin{proof}
Let $f \colon F \to G$ be an objectwise monomorphism in $\Fun(\scC, \sSet)$.
By Proposition~\ref{st:alternative presentation of r_C^*}, it suffices to prove the same claim for $s_\scC^*$.
Let $(\alpha, x), (\beta, y) \in s_\scC^* F$ and suppose that $(s_\scC^*f)(\alpha,x) = (s_\scC^*f)(\beta, y)$.
This readily implies $\alpha = \beta$.
The data $x = (x_{i,j})_{i,j = 0, \ldots, n}$ and $y = (y_{i,j})_{i,j = 0, \ldots, n}$ are sections of $\Sigma \alpha^*F$ (in the notation of~\eqref{eq:Sigma alpha}).
Under $s_\scC^*f$, these are sent to the sections $f(x) = (f(x_{i,j}))_{i,j = 0, \ldots, n}$ and $f(y) = (f(y)_{i,j})_{i,j = 0, \ldots, n}$ of $\Sigma \alpha^*G$, which are equal by our assumption.
Since $f$ was an objectwise monomorphism, it follows that $x = y$.
\end{proof}

The following observation is crucial:
it will allow us to enhance $r_\scC^*$ to a functor taking values in the $\infty$-category of functors with values in Cisinski's $\infty$-category of spaces (see Definition~\ref{def:oo-Cat of spaces} below).
We give a short proof for completeness.

\begin{lemma}
\label{st:r^* and change of index category}
\emph{\cite[Rmk.~3.2.5.7]{Lurie:HTT}}
Let $\psi \colon \scD \to \scC$ and $F \colon \scC \to \sSet$ be functors of (1-)categories.
\begin{myenumerate}
\item There is a pullback square of simplicial sets
\begin{equation}
\label{eq:r^*_[-] yields pb squares}
\begin{tikzcd}
	r_\scD^*(\psi^*F) \ar[r] \ar[d] & r_\scC^*F \ar[d]
	\\
	N\scD \ar[r, "N\psi"'] & N\scC
\end{tikzcd}
\end{equation}
That is, there is a canonical natural isomorphism
\begin{equation}
	r_\scD^* \circ \psi^* \cong (N\psi)^* \circ r_\scC^*\,.
\end{equation}

\item For any $\alpha \colon [n] \to \scD$, there is a canonical bijection
\begin{equation}
	(r_\scC^*F)_{\psi_*\alpha,n} \cong \big( r_\scD^*(\psi^*F) \big)_{\alpha,n}\,.
\end{equation}
\end{myenumerate}
\end{lemma}

\begin{proof}
An $l$-simplex of $r_\scD^*(\psi^*F)$ consists of a pair $(\alpha,\eta)$, where $\alpha \colon [l] \to \scD$ is an $l$-simplex of $N\scD$, and where $\eta \colon \Delta^\cdot \circ \pr \longrightarrow F \circ \psi \circ \alpha \circ \tau_{[l]}$ is a natural transformation of functors $\bbDelta_{/[l]} \to \sSet$ (in the notation from before~\eqref{eq:r_C^* explicit}).
This is equivalently a triple $(\alpha, \eta', \beta)$ of an $l$-simplex $\alpha$ in $N\scD$, an $l$-simplex $\beta$ of $N\scC$ and a natural transformation $\eta' \colon \Delta^\cdot \circ \pr \longrightarrow F \circ \beta \circ \tau_{[l]}$, such that $\beta = \psi_*\alpha$.
This identification establishes that diagram~\eqref{eq:r^*_[-] yields pb squares} is a pullback square in $\sSet$.
Part~(2) can either be seen explicitly by unravelling the definition of $r_\scD^*F$, or by applying the pasting law for pullbacks to Part~(1).
\end{proof}

\subsection{Homotopy-theoretic properties of the rectification functor}

In this section we collect various homotopy-theoretic properties of the rectification functor $r_\scC^*$ (see Definition~\ref{def:rectification functor}).
Given a $\rmU$-small category $\scC$, we endow $\Fun(\scC, \sSet)$ with the projective model structure.
The following statement is a combination of~\cite[Thm.~C, Prop.~5.1, Prop.~5.2]{HM:Left_fibs_and_hocolims}:

\begin{theorem}
\label{st:r_! r^* QEq}
There is a pair of Quillen equivalences
\begin{equation}
\begin{tikzcd}[column sep=1.25cm]
	\Fun(\scC, \sSet) \ar[r, shift left=0.15cm, "h_{\scC!}", "\perp"' yshift=0.05cm]
	& \sSet_{/N\scC}
	\ar[r, shift left=0.15cm, "r_{\scC!}", "\perp"' yshift=0.05cm]
	\ar[l, shift left=0.15cm, "h_\scC^*"]
	& \Fun(\scC, \sSet)\,. \ar[l, shift left=0.15cm, "r_\scC^*"]
\end{tikzcd}
\end{equation}
On the level of total derived functors between homotopy categories, the functors $\bbL r_{\scC!}$ and $\bbL h_{\scC!}$, and hence also the functors $\bbR r_\scC^*$ and $\bbR h_\scC^*$, are weak mutual inverses.
\end{theorem}

The proof that the Quillen adjunctions in Theorem~\ref{st:r_! r^* QEq} are in fact Quillen equivalences is established by showing that the left derived functors of $h_{\scC!}$ and $r_{\scC!}$ are weak mutual inverses.
This relies on the following result:

\begin{proposition}
\label{st:nat weqs for r_! and h_!}
\emph{\cite[Props.~5.1 and~5.2]{HM:Left_fibs_and_hocolims}}
Let $X \in \Fun(\scC, \sSet)$ and $A \in \sSet_{/N\scC}$.
\begin{myenumerate}
\item There exists a natural transformation $\tau_X \colon r_{\scC!} h_{\scC!} X \to X$, which is a weak equivalence in $\Fun(\scC, \sSet)$ whenever $X$ is projectively cofibrant.

\item There exists a natural zig-zag of weak equivalences $A \leftarrow L(A) \rightarrow h_{\scC!} r_{\scC!} A$ in $\sSet_{/N\scC}$.
\end{myenumerate}
\end{proposition}

Heuts-Moerdijk moreover prove the following consequence of Theorem~\ref{st:r_! r^* QEq}, which relates to the Kan-Quillen model structure rather than the covariant model structure:

\begin{theorem}
\label{st:h_! -| h^* QEq localised}
\emph{\cite[Cor.~D]{HM:Left_fibs_and_hocolims}}
Let $\scC$ be a small category, and let $L^B_\scC \Fun(\scC,\sSet)$ denote the left Bousfield localisation of the projective model structure on $\Fun(\scC,\sSet)$ at all morphisms in $\scC$ (or, more precisely, at their images under the coYoneda embedding).
The Quillen equivalence $h_{\scC!} \dashv h_\scC^*$ from Theorem~\ref{st:r_! r^* QEq} induces a Quillen equivalence
\begin{equation}
\begin{tikzcd}[column sep=1.2cm]
	L^B_\scC \Fun(\scC, \sSet) \ar[r, shift left=0.15cm, "h_{\scC!}", "\perp"' yshift=0.05cm]
	& \big( \sSet_{/N\scC} \big)_{KQ}\,, \ar[l, shift left=0.15cm, "h_\scC^*"]
\end{tikzcd}
\end{equation}
where $\sSet_{/N\scC}$ is endowed with the slice model structure induced from the Kan-Quillen model structure on $\sSet$.
\end{theorem}

An analogue of this statement also holds true for the right-hand Quillen equivalence $r_{\scC!} \dashv r_\scC^*$ of Theorem~\ref{st:r_! r^* QEq}:

\begin{proposition}
\label{st:r_! r^* localised}
The Quillen equivalence $r_{\scC!} \dashv r_\scC^*$ from Theorem~\ref{st:r_! r^* QEq} induces a Quillen equivalence
\begin{equation}
\begin{tikzcd}[column sep=1.2cm]
	\big( \sSet_{/N\scC} \big)_{KQ} \ar[r, shift left=0.15cm, "r_{\scC!}", "\perp"' yshift=0.05cm]
	& L^B_\scC \Fun(\scC, \sSet)\,. \ar[l, shift left=0.15cm, "r_\scC^*"]
\end{tikzcd}
\end{equation}
\end{proposition}

\begin{proof}
Let $h^{(-)} \colon \scC^\opp \to \Fun(\scC, \sSet)$, $c \mapsto h^c = \scC(c,-)$ denote the co-Yoneda embedding of $\scC$.
Let $S$ denote the set of morphisms $h^g \colon h^d \to h^c$ in $\Fun(\scC, \sSet)$, where $g \colon c \to d$ ranges over all morphisms in $\scC$.
Note that each $h^c \in \Fun(\scC, \sSet)$ is projectively cofibrant.
It is shown in~\cite[Prop.~6.1]{HM:Left_fibs_and_hocolims} and the proof of Cor.~D in~\cite{HM:Left_fibs_and_hocolims} that
\begin{equation}
	\big( \sSet_{/N\scC} \big)_{KQ}
	= L^B_{h_{\scC!}S} \big( \sSet_{/N\scC} \big)
\end{equation}
as model categories.
We first show that the Quillen adjunction $r_{\scC!} : \sSet_{/N\scC} \rightleftarrows \Fun(\scC, \sSet) : r_\scC^*$ descends to a Quillen adjunction
\begin{equation}
\begin{tikzcd}[column sep=1.2cm]
	r_{\scC!} : L^B_{h_{\scC!} S} \big( \sSet_{/N\scC} \big) \ar[r, shift left=0.15cm, "\perp"' yshift=0.05cm]
	& L^B_\scC \Fun(\scC, \sSet) : r_\scC^*\,. \ar[l, shift left=0.15cm]
\end{tikzcd}
\end{equation}
By~\cite[Prop.~3.3.18]{Hirschhorn:MoCats}, it suffices to show that $r_{\scC!}$ takes each morphism $h^g \colon h^d \to h^c$ in $S$ to a weak equivalence in $L^B_\scC \Fun(\scC,\sSet)$.
To see this, let $X \in L^B_\scC \Fun(\scC, \sSet)$ be fibrant.
By Proposition~\ref{st:nat weqs for r_! and h_!} there exists a natural transformation $\tau_Y \colon r_{\scC!} h_{\scC!} Y \to Y$, which is a projective weak equivalence whenever $Y \in \Fun(\scC, \sSet)$ is projectively cofibrant.
This induces a commutative diagram of simplicial sets (recall that we denote the simplicial hom sets in a simplicially enriched category $\scM$ by $\ul{\scM}(-,-)$)
\begin{equation}
\begin{tikzcd}[column sep=2.5cm, row sep=1cm]
	\ul{\Fun(\scC, \sSet)}(r_{\scC!} h_{\scC!} h^c, X)
	\ar[r, "{(r_{\scC!} h_{\scC!} h^g)^*}"] \ar[d, "\tau_{(h^c)}^*"']
	& \ul{\Fun(\scC, \sSet)}(r_{\scC!} h_{\scC!} h^d, X)
	\ar[d, "\tau_{(h^d)}^*"]
	\\
	X(c) \ar[r, "{X(g)}"']
	& X(d)
\end{tikzcd}
\end{equation}
The vertical morphisms are weak equivalences because $X$ is projectively fibrant and corepresentables are cofibrant in $\Fun(\scC, \sSet)$.
The bottom morphism is a weak equivalence as well since $X$ is fibrant in $L^B_\scC \Fun(\scC, \sSet)$; that is, $X(g)$ is a weak equivalence for each morphism $g$ in $\scC$.
It thus follows that the top horizontal morphism is a weak equivalence, and hence, that the Quillen adjunction $r_{\scC!} \dashv h_{\scC!}$ descends to the localisations as claimed.

It remains to show that this descended adjunction is even a Quillen equivalence.
However, since the natural weak equivalences in Proposition~\ref{st:nat weqs for r_! and h_!} are still weak equivalences in the localisations, the fact that the left derived functors of $h_{\scC!}$ and $r_{\scC!}$ are mutual weak inverses still holds true at the level of the localisations.
\end{proof}

\begin{remark}
Alternatively, one can employ a different way of presenting $(\sSet_{/N\scC})_{KQ}$ as a Bousfield localisation of the covariant model structure on $\sSet_{/N\scC}$:
it is shown in~\cite[Lemma~6.1]{HM:Left_fibs_and_hocolims} and the proof of~\cite[Cor.~D]{HM:Left_fibs_and_hocolims} that $(\sSet_{/N\scC})_{KQ}$ agrees with the left Bousfield localisation of the covariant model structure at the morphisms
\begin{equation}
\begin{tikzcd}
	\Delta^{\{1\}} \ar[rr, hookrightarrow] \ar[dr, "y"']
	& & \Delta^1 \ar[dl, "f"]
	\\
	& N \scC &
\end{tikzcd}
\end{equation}
where $f$ ranges over all morphisms $f \colon f_0 \to f_1$ in $\scC$.
Note that $r_{\scC!} f_i = h^{f_i} = \scC(f_i,-)$ is the functor corepresented by $f_i$.
One thus needs to show that the morphism $h^{f_1} = r_{\scC!}(f_1) \to r_{\scC!} f$ is a weak equivalence in $L^B_\scC \Fun(\scC, \sSet)$.
Let $X \in L^B_\scC \Fun(\scC, \sSet)$ be fibrant.
We have a triangle
\begin{equation}
\begin{tikzcd}[column sep={4cm,between origins}]
	& \ul{\Fun(\scC, \sSet)} (r_{\scC!}f, X) \ar[dl] \ar[dr] &
	\\
	\ul{\Fun(\scC, \sSet)} (r_{\scC!}f_0, X)
	\cong X(f_0) \ar[rr]
	& & \ul{\Fun(\scC, \sSet)} (r_{\scC!}f_1, X)
	\cong X(f_1)
\end{tikzcd}
\end{equation}
and one can show that this commutes \textit{up to simplicial homotopy}.
The left-hand and bottom morphisms are weak equivalences between Kan complexes, and thus homotopy equivalences.
It then follows that the right-hand morphism is a homotopy equivalence as well, showing that $r_{\scC!} \dashv r_\scC^*$ descends to a Quillen adjunction between the localisations.
\qen
\end{remark}

\begin{remark}
The fact that $r_{\scC!} \dashv r_\scC^*$ descends to a Quillen adjunction between the localisations has been sketched without proof in~\cite{HM:Left_fibs_and_hocolims} (see the paragraph after Cor.~D).
We give a proof here for completeness.
\qen
\end{remark}

Let $\Kan \subset \sSet$ denote the full subcategory on the Kan complexes.

\begin{corollary}
\label{st:weakly constant diags and KanFibs}
Let $\scC$ be a small category, and let $F \colon \scC \to \Kan$ be a functor which sends each morphism in $\scC$ to a weak equivalence in $\sSet$.
Then, the morphism $r_\scC^*F \to N\scC$ is a Kan fibration.
\end{corollary}

\begin{proof}
A functor $F$ with these properties is precisely a fibrant object in $L^B_\scC \Fun(\scC,\sSet)$.
\end{proof}

\section{From strict diagrams in $\sSet$ to coherent diagrams of spaces}
\label{sec:from strict to coherent diagrams}

In this section we start by recalling Cisinski's $\infty$-category of spaces $\scS$.
We describe in detail the data encoding maps from simplicial sets to $\scS$.
Combining this with the rectification functor $r_\scC^*$ from Section~\ref{sec:r^*} we build functors of $\infty$-categories $\gamma \colon N\Kan \to \scS$ and $\gamma_\scC \colon N\Fun(\scC, \Kan) \longrightarrow \scFun(N\scC, \scS)$, for any $\rmU$-small 1-category $\scC$.
Given a functor $F \colon \scC \to \Kan$, we show that the $\infty$-functor $\gamma_\scC F \colon N\scC \to \scS$ classifies the left fibration $r_\scC^*F \to N\scC$.

\begin{definition}
\label{def:oo-Cat of spaces}
\emph{\cite[Def.~5.2.3]{Cisinski:HCats_HoAlg}, \cite{Cisinski:HCats_HoAlg_Errata}}
We define the following two $\infty$-categories:
\begin{myenumerate}
\item The \emph{$\infty$-category $\scS$ of ($\rmU$-small) spaces} is the simplicial set given as follows:
an $n$-simplex of $\scS$ consists of a left fibration $p \colon X \to \Delta^n$ with $\rmU$-small fibres, together with a choice, for each $f \colon \Delta^m \to \Delta^n$, of a pullback square of $\rmU$-small simplicial sets
\begin{equation}
\begin{tikzcd}
	f^*X \ar[r] \ar[d] & X \ar[d]
	\\
	\Delta^m \ar[r, "u"'] & \Delta^n
\end{tikzcd}
\end{equation}
such that the pullback assigned to the identity on $\Delta^n$ is the identity on $X$.
The simplicial structure of $\scS$ arises from the pasting law for pullbacks

\item The \emph{$\infty$-category} of pointed ($\rmU$-small) spaces is the simplicial set $\scS_*$ whose $n$-simplices consist of a left fibration $p \colon X \to \Delta^n$ with $\rmU$-small fibres, together with a choice of a section $s$ of $p$ and, for each $u \colon \Delta^m \to \Delta^n$, of a pullback square of $\rmU$-small simplicial sets as above.
The simplicial structure of $\scS_*$ arises from the pasting law for pullbacks
\end{myenumerate}
\end{definition}

We denote the morphism which forgets the sections by $p_\univ \colon \scS_* \to \scS$.
We recall the following two key facts about the $\infty$-categories $\scS$ and $\scS_*$:

\begin{lemma}
\label{st:maps to scS}
\emph{\cite[Prop.~5.2.4]{Cisinski:HCats_HoAlg}}
Specifying a cartesian square
\begin{equation}
\begin{tikzcd}
	Y \ar[r, "\tilde{F}"] \ar[d, "f"']
	& \scS_* \ar[d, "p_\univ"]
	\\
	X \ar[r, "F"']
	& \scS
\end{tikzcd}
\end{equation}
where $f$ is a left fibration, is equivalent to choosing a cartesian square
\begin{equation}
\begin{tikzcd}
	\varphi^*X \ar[r, "\tilde{\varphi}"] \ar[d, "\varphi^*f"']
	& X \ar[d, "f"]
	\\
	\Delta^n \ar[r, "\varphi"']
	& Y
\end{tikzcd}
\end{equation}
for each map $\varphi \colon \Delta^n \to Y$, where $\varphi^*X$ is $\rmU$-small.
\end{lemma}

\begin{definition}
\cite[Def.~5.2.5]{Cisinski:HCats_HoAlg}
In the situation of Lemma~\ref{st:maps to scS}, we say that \emph{the morphism $F$ classifies the left fibration $f$}.
\end{definition}

\begin{proposition}
\emph{\cite[Cor.~5.2.8]{Cisinski:HCats_HoAlg}}
The morphism $p_\univ \colon \scS_* \to \scS$ is a left fibration.
Every left fibration $f \colon Y \to X$ in $\sSet$ is classified by some map $F \colon X \to \scS$.
\end{proposition}

\begin{remark}
\label{rmk:mps L --> scS}
Let $L$ be a $\rmU$-small simplicial set.
Explicitly, each map of ($\rmV$-small) simplicial sets $F \colon L \to \scS$ is specified as follows:

\begin{myitemize}
\item \textit{The map $F_n \colon L_n \to \scS_n$:}
to each $n$-simplex $a \colon \Delta^n \to L$ of $L$, the map $F$ assigns an $n$-simplex $F_n(a)$ of $\scS$, i.e.~a left fibration $A(a) \to \Delta^n$ together with choices of pullback squares
\begin{equation}
\begin{tikzcd}
	A(a;v) \ar[r] \ar[d]
	& A(a) \ar[d]
	\\
	\Delta^r \ar[r, "v"']
	& \Delta^n
\end{tikzcd}
\end{equation}
for each morphism $v \colon \Delta^r \to \Delta^n$ in $\sSet$.

\item \textit{Compatibility with the simplicial structure maps of $L$ and $\scS$:}
in order for $F$ to define a morphism of $(\rmV$-small) simplicial sets, these data have to satisfy the following consistency condition:
given a morphism $u \colon \Delta^k \to \Delta^n$, we obtain the $k$-simplex $u^*a$ of $L$ as the composition $a \circ u \colon \Delta^k \to L$.
We also obtain a $k$-simplex $u^*F(a)$ in $\scS$, and we spell out what it means for these to agree.
The assignment $F$ associates to $u^*a \in L_k$ a left fibration $A(u^*a) \to \Delta^k$ and choices of pullback squares
\begin{equation}
\begin{tikzcd}
	A(u^*a;v) \ar[r] \ar[d]
	& A(u^*a) \ar[d]
	\\
	\Delta^r \ar[r, "v"']
	& \Delta^k
\end{tikzcd}
\end{equation}
for maps $v \colon \Delta^r \to \Delta^k$.
The left fibration $A(u^*a) \to \Delta^k$ has to coincide with the left vertical morphism in the chosen pullback
\begin{equation}
\begin{tikzcd}
	A(a;u) \ar[r] \ar[d]
	& A(a) \ar[d]
	\\
	\Delta^k \ar[r, "v"']
	& \Delta^n
\end{tikzcd}
\end{equation}
for the $n$-simplex $F(a)$, i.e.~$A(u^*a) = A(a;u)$.
Further, the chosen pullbacks of $F(a)$ and $u^*F(a)$ must agree, which we may describe by the equation
\begin{equation}
	A(u^*a;v) = A(a, u \circ v)
\end{equation}
for each map $v \colon \Delta^r \to \Delta^k$.
\qen
\end{myitemize}
\end{remark}

\begin{remark}
\label{rmk:mps L --> scS as cart nat trafos}
An equivalent way to describe a morphism $F \colon L \to \scS$ is as follows:
Consider the comma category $\bbDelta_{/L}$ and the functor $ \Delta^\cdot \circ \pr \colon \bbDelta_{/L} \to \bbDelta \to \sSet$ (where $\pr$ is the canonical projection functor).
The data from Remark~\ref{rmk:mps L --> scS} are equivalent to that of a functor $A \colon \bbDelta_{/L} \to \sSet$, together with a natural transformation $p \colon A \to \Delta^\cdot \circ \pr$ satisfying that
\begin{myenumerate}
\item $p$ is a \textit{cartesian} natural transformation (i.e.~all its naturality squares are cartesian) and

\item $p$ consists of left fibrations.
\qen
\end{myenumerate}
\end{remark}

\begin{lemma}
\label{st:mps L --> scS from LFibs}
Let $A \to L$ be a left fibration of $\rmU$-small simplicial sets.
Suppose we are given choices of pullback squares
\begin{equation}
\begin{tikzcd}
	A(a) \ar[r] \ar[d]
	& A \ar[d]
	\\
	\Delta^n \ar[r, "a"']
	& L
\end{tikzcd}
\end{equation}
for each map $a \colon \Delta^n \to L$.
These data induce a consistent family of left fibrations and choices of pullbacks as in Remark~\ref{rmk:mps L --> scS}, i.e.~the data of a morphism $L \to \scS$.
\end{lemma}

\begin{proof}
We need to provide, for each morphism $v \colon \Delta^r \to \Delta^n$, a choice of pullback square
\begin{equation}
\begin{tikzcd}
	A(a;v) \ar[r] \ar[d]
	& A(a) \ar[d]
	\\
	\Delta^r \ar[r, "v"']
	& \Delta^n
\end{tikzcd}
\end{equation}
However, we can augment any such diagram to a double square
\begin{equation}
\begin{tikzcd}[column sep={2.25cm,between origins}]
	A(a;v) \ar[r] \ar[d]
	& A(a) \ar[d] \ar[r]
	& A \ar[d]
	\\
	\Delta^r \ar[r, "v"']
	& \Delta^n \ar[r, "a"']
	& L
\end{tikzcd}
\end{equation}
Note that, by assumption, we have already specified a pullback
\begin{equation}
\begin{tikzcd}
	A(a \circ v) \ar[r] \ar[d]
	& A \ar[d]
	\\
	\Delta^r \ar[r, "a \circ v"']
	& L
\end{tikzcd}
\end{equation}
for the outer square.
Thus, we see that setting $A(a;v) = A(a \circ v)$ for each $v \colon \Delta^r \to \Delta^n$ produces a consistent family as required.
\end{proof}

\begin{theorem}
\label{st:functor NKan -> S}
The functors
\begin{equation}
	\big\{ r_{[n]}^* \colon \Fun([n], \sSet) \to \sSet_{/\Delta^n} \, \big|\, n \in \NN_0 \big\}
\end{equation}
assemble into a functor of $\infty$-categories
\begin{equation}
	\gamma \colon N \Kan \longrightarrow \scS\,.
\end{equation}
\end{theorem}

\begin{proof}
The functor $\gamma$ sends an $n$-simplex $X \colon [n] \to \Kan$ of $N \Kan$ to the following $n$-simplex of $\scS$:
its underlying left fibration is given by \smash{$r_{[n]}^*X \to \Delta^n$}.
Given any functor $u \colon [k] \to [n]$, we associate to this the square
\begin{equation}
\begin{tikzcd}
	r_{[k]}^*(u^*X) \ar[r] \ar[d] & r_{[n]}^*X \ar[d]
	\\
	\Delta^k \ar[r, "u"'] & \Delta^n
\end{tikzcd}
\end{equation}
which is cartesian by Lemma~\ref{st:r^* and change of index category}.
(recall that for a functor $u \colon [k] \to [n]$ we denote the map $\Delta^k \to \Delta^n$ induced by the Yoneda embedding of $\bbDelta$ again by $u$).
In other words, the left fibrations and choices of pullbacks are generated under Lemma~\ref{st:mps L --> scS from LFibs} from the left fibration \smash{$r_{[k]}^*X \to \Delta^k$} and the choices of pullbacks as in the above diagram.
Thus, these data constitute an $n$-simplex of $\scS$.
\end{proof}

\begin{theorem}
\label{st:LFib classified by gamma o NX}
Let $\scC$ be a $\rmU$-small category, and let $X \colon \scC \to \Kan$ be a functor.
The diagram
\begin{equation}
\begin{tikzcd}
	r_\scC^*X \ar[rr] \ar[d]
	& & \scS_* \ar[d, "p_\univ"]
	\\
	N \scC \ar[r, "NX"']
	& N \Kan \ar[r, "\gamma"']
	& \scS
\end{tikzcd}
\end{equation}
is cartesian (in $\rmV\sSet$).
That is, the the functor $NX \circ \gamma \colon N\scC \to \scS$ classifies the left fibration $r_\scC^*X \to N\scC$.
\end{theorem}

This provides a simplified analogue of~\cite[3.2.5.21]{Lurie:HTT}.

\begin{proof}
We have to show that there is an isomorphism of ($\rmV$-small) simplicial sets
\begin{equation}
	r_\scC^*X \cong N \scC \underset{\scS}{\times} \scS_*\,.
\end{equation}
An $n$-simplex in $N \scC \times_\scS \scS_*$ consists of the following data:
first, we have an $n$-simplex of $N_n \scC$, or equivalently a functor $\alpha \colon [n] \to \scC$.
The image of $\alpha$ in $\scS$ under the functor $\gamma \circ NX$ is, by construction of $\gamma$, the left fibration \smash{$q_\alpha \colon r_{[n]}^*(\alpha^*X) \longrightarrow \Delta^n$}, together with the choices of pullbacks induced by Lemma~\ref{st:r^* and change of index category}.
An $n$-simplex in $N \scC \times_\scS \scS_*$ then consists of a pair $(\alpha, s)$, where $s$ is a section of $q_\alpha$.
By the above and Lemma~\ref{st:r^* and change of index category} we have a commutative diagram
\begin{equation}
\begin{tikzcd}
	\Delta^n \ar[ddr, bend right=25, "1_{\Delta^n}" description] \ar[rrd, dashed, bend left=25, "a" description] \ar[rd, "s" description]
	& 
	& 
	\\
	& r_{[n]}^* (\alpha^*X) \ar[r] \ar[d, "q_\alpha"]
	&r_\scC^*X \ar[d, "q"]
	\\
	& \Delta^n \ar[r, "N \alpha"']
	& N\scC
\end{tikzcd}
\end{equation}
in $\sSet$.
The top morphism now defines an $n$-simplex in $r_\scC^*X$ from the pair $(\alpha,s)$.
This defines a map on $n$-simplices
\begin{equation}
	\big( (N\scC) \times_\scS \scS_* \big)_n \longrightarrow \big( r_\scC^*X \big)_n\,.
\end{equation}

In the opposite direction, denote the left fibration \smash{$r_\scC^*X \to N\scC$} by $q$.
Given an $n$-simplex \smash{$a \colon \Delta^n \to r_\scC^*X$}, set $\alpha' \coloneqq q \circ a$.
This is the nerve of a \textit{unique} functor $\alpha \colon [n] \to \scC$ (since the nerve functor $N \colon \Cat \to \sSet$ is fully faithful as a functor of 1-categories).
Consider the diagram
\begin{equation}
\begin{tikzcd}
	\Delta^n \ar[ddr, bend right=25, "1_{\Delta^n}" description] \ar[rrd, bend left=25, "a" description] \ar[rd, dashed, "s" description]
	& 
	& 
	\\
	& r_{[n]}^* (\alpha^*X) \ar[r] \ar[d, "q_\alpha"]
	&r_\scC^*X \ar[d, "q"]
	\\
	& \Delta^n \ar[r, "N \alpha"']
	& N\scC
\end{tikzcd}
\end{equation}
By Lemma~\ref{st:r^* and change of index category} the square in this diagram is cartesian, so that there is a \textit{unique} morphism $s$ fitting into this diagram as the dashed morphism.
We thus obtain a map
\begin{equation}
	\big( r_\scC^*X \big)_n \longrightarrow \big( (N\scC) \times_\scS \scS_* \big)_n\,.
\end{equation}
The universal property of pullbacks ensures that this is the inverse to the above map on $n$-simplices.
Finally, it follows from the pasting law for pullbacks and Lemma~\ref{st:r^* and change of index category} that these maps on $n$-simplices form maps of simplicial sets.
\end{proof}

Next, we enhance Theorem~\ref{st:functor NKan -> S} from providing a functor between $N\Kan$ and the $\infty$-category of spaces $\scS$ so that it provides functors between diagrams in both $\infty$-categories.

\begin{remark}
\label{rmk:splcs in Fun(L, scS)}
Let $L$ be a $\rmU$-small simplicial set.
We can describe the simplicial set $\scFun(L, \scS)$ explicitly as follows:
an $n$-simplex in $\scFun(L, \scS)$ is a morphism $F \colon \Delta^n {\times} L \to \scS$.
Such a morphism corresponds to specifying, for each morphism $a \colon \Delta^k \to \Delta^n {\times} L$, a left fibration $A(a) \to \Delta^k$ with further choices of pullbacks $A(a;v)$ for each morphism $v \colon \Delta^r \to \Delta^k$ (using the notation of Remark~\ref{rmk:mps L --> scS}), satisfying the consistency condition $A(a;v) = A(a \circ v)$ from Remark~\ref{rmk:mps L --> scS}.
We denote the data corresponding to the $n$-simplex $F$ by $\{A\}$.

Next, we describe the action of the simplicial structure maps of $\scFun(L, \scS)$.
Consider a morphism $u \colon \Delta^l \to \Delta^n$; this defines a simplicial structure map $(u {\times} 1_L)^* \colon \scFun(L, \scS)_n \to \scFun(L, \scS)_l$.
It acts by precomposing morphisms $F \colon \Delta^n {\times} L \to \scS$ with the morphism $(u {\times} 1_L) \colon \Delta^l {\times} L \to \Delta^n {\times} L$.
Given the above $n$-simplex $F$, specified by the consistent collection of left fibrations and chosen pullbacks $\{A\}$, we describe its image under $(u {\times} 1_L)^*$ explicitly:
for a morphism $b \colon \Delta^k \to \Delta^l {\times} L$, the associated left fibration of the new $l$-simplex is given by $A'(b) = A((u {\times} 1_L) \circ b) \to \Delta^k$, with the necessary choices of pullback squares induced in the same manner.
This gives a new consistent family of left fibrations and pullback squares $\{A'\}$ which corresponds to the morphism $(u {\times} 1_L)^*F \colon \Delta^l {\times} L \to \scS$.
\qen
\end{remark}

We have the following enhancement of Theorem~\ref{st:functor NKan -> S}:

\begin{theorem}
\label{st:functor NFun(C, Kan) -> Fun(NC, S)}
Let $\scC$ be a small category.
The functors
\begin{equation}
	\big\{ r_{[n] {\times} \scC}^* \colon \Fun \big( [n] {\times} \scC, \sSet \big) \longrightarrow \sSet_{/\Delta^n \times N\scC} \, \big|\, n \in \NN_0 \big\}
\end{equation}
assemble into a functor between $\infty$-categories
\begin{equation}
	\gamma_\scC \colon N \Fun(\scC, \Kan) \longrightarrow \scFun(N\scC, \scS)\,.
\end{equation}

\end{theorem}

\begin{proof}
There are canonical isomorphisms
\begin{equation}
	N_n \Fun(\scC, \Kan) \cong \Fun \big( [n] {\times} \scC, \Kan \big)
	\qandq
	\big( \scFun(N\scC, \scS) \big)_n \cong \rmV\sSet(\Delta^n {\times} N\scC, \scS)\,.
\end{equation}
Let $X \colon [n] \to \Fun(\scC, \Kan)$ be an $n$-simplex in $N \Fun(\scC,\Kan)$.
Equivalently, we can write these data as a functor $X^\dashv \colon [n] {\times} \scC \to \sSet$.
We produce from this a consistent family $\{A\}$ of left fibrations and choices of pullbacks as in Remark~\ref{rmk:splcs in Fun(L, scS)} (recall that these families correspond to $n$-simplices in $\scFun(N\scC, \scS)$).
This will be aided by Lemma~\ref{st:mps L --> scS from LFibs}:
first, consider the left fibration \smash{$r_{[n] {\times} \scC}^* X^\dashv \longrightarrow \Delta^n {\times} N\scC$}.
Let $\alpha \colon [k] \to [n] {\times} \scC$ be any functor.
We assign to this the square
\begin{equation}
\begin{tikzcd}
	A(a) = r_{[k]}^*(\alpha^*X^\dashv) \ar[r] \ar[d]
	& r_{[n] {\times} \scC}^* X^\dashv = A \ar[d]
	\\
	\Delta^k \ar[r, "N\alpha"']
	& \Delta^n {\times} N\scC
\end{tikzcd}
\end{equation}
This is a pullback square by Lemma~\ref{st:r^* and change of index category}.
From the left fibration \smash{$r_{[n] {\times} \scC}^* X^\dashv \longrightarrow \Delta^n {\times} N\scC$} and these choices of pullback squares, Lemma~\ref{st:mps L --> scS from LFibs} produces a consistent family $\{A\}$ as required, i.e.~a morphism $\Delta^n {\times} N\scC \to \scS$ and thus, equivalently, an $n$-simplex in $\scFun(N\scC, \scS)$.

It remains to show that this assignment $X \longmapsto \{A\}$ defines a morphism $N\Fun(\scC, \Kan) \longrightarrow \scFun(N\scC, \scS)$ of simplicial sets; that is, we need to show that it is compatible with the simplicial structure maps.
Let $u \colon [l] \to [n]$ be a morphism in $\bbDelta$.
On the domain side, the structure map associated to $u$ sends an $n$-simplex $X^\dashv \colon [n] {\times} \scC \to \Kan$ to the $l$-simplex $(u {\times} 1_\scC)^*X^\dashv \colon [l] {\times} \scC \to \Kan$.
The coherent family associated to this by the above construction arises (via Lemma~\ref{st:mps L --> scS from LFibs}) from the left fibration
\begin{equation}
	r_{[l] {\times} \scC}^* \big( (u {\times} 1_\scC)^* X^\dashv \big)
	\longrightarrow \Delta^l {\times} N\scC\,,
\end{equation}
together with the choices of pullbacks
\begin{equation}
\begin{tikzcd}[column sep=1.5cm]
	r_{[k]}^*(\beta^* (u {\times} 1_\scC)^* X^\dashv) \ar[r] \ar[d]
	& r_{[l] {\times} \scC}^* \big( (u {\times} 1_\scC)^* X^\dashv \big) \ar[d]
	\\
	\Delta^k \ar[r, "N \beta"']
	& \Delta^l {\times} N\scC
\end{tikzcd}
\end{equation}
for each functor $\beta \colon [k] \to [l] {\times} \scC$.
Invoking the explicit formulas for the action of the simplicial structure maps in $\scFun(N\scC, \scS)$ from Remark~\ref{rmk:splcs in Fun(L, scS)}, we see that this agrees with the result of applying the structure map associated to $u$ on the codomain side to the consistent family $\{A\}$ constructed above.
\end{proof}

\begin{theorem}
\label{st:gamma_C = gamma o N o dashv}
Let $\scC$ be a small category, and let $X \colon \scC \to \Kan$ be a functor.
There is a strictly commutative diagram of categories and functors
\begin{equation}
\begin{tikzcd}[column sep=1cm, row sep=0.75cm]
	N\Fun(\scC, \Kan) \ar[d, "\gamma_\scC"'] \ar[r, "\cong"]
	& \scFun(N\scC, N\Kan) \ar[d, "\gamma \circ (-)"]
	\\
	\scFun(N\scC, \scS) \ar[r, equal]
	& \scFun(N\scC, \scS)
\end{tikzcd}
\end{equation}
\end{theorem}

\begin{proof}
Let $X \colon \Delta^n \to N\Fun(\scC, \Kan)$ be an $n$-simplex of $N\Fun(\scC, \Kan)$.
Under the tensor-hom adjunction, this corresponds to a unique functor $X^\dashv \colon [n] {\times} \scC \to \Kan$.
The image of this $n$-simplex under the functor $\gamma_\scC$ is the $n$-simplex
\begin{equation}
	\gamma_\scC(X) \in \scFun(N\scC, \scS)_n
	\cong \rmV\sSet(\Delta^n {\times} N\scC, \scS)
\end{equation}
which is induced via Lemma~\ref{st:mps L --> scS from LFibs} from the left fibration \smash{$r_{[n] \times \scC}^*X^\dashv \longrightarrow \Delta^n {\times} N\scC$} and the choices of pullbacks \smash{$A(N\alpha) = r_{[k]}^*(\alpha^*X^\dashv) \to \Delta^k$}, for functors $\alpha \colon [k] \to [n] {\times} \scC$ (in the notation of Remark~\ref{rmk:mps L --> scS}; see also the proof of Theorem~\ref{st:functor NFun(C, Kan) -> Fun(NC, S)}).
More explicitly, $\gamma_\scC (X)$ is a morphism of simplicial sets $\Delta^n {\times} N\scC \to \scS$; it sends a $k$-simplex $\alpha \colon [k] \to [n] {\times} \scC$ to the $k$-simplex of $\scS$ which is given by the left fibration \smash{$r_{[k]}^*(\alpha^*X^\dashv) \longrightarrow \Delta^k$} together with the choices of pullbacks induced by Lemma~\ref{st:r^* and change of index category}.

We compare this to the $n$-simplex of $\scFun(N\scC, \scS)$ corresponding to the map
\begin{equation}
\begin{tikzcd}
	\Delta^n \times N\scC \cong N([n] \times \scC) \ar[r, "N X^\dashv"]
	& N\Kan \ar[r, "\gamma"]
	& \scS\,.
\end{tikzcd}
\end{equation}
Using Remark~\ref{rmk:splcs in Fun(L, scS)} and the construction of the functor $\gamma$ in (the proof of) Theorem~\ref{st:functor NKan -> S}, we can describe the action of this second map on $k$-simplices explicitly:
consider a functor $\alpha \colon [k] \to [n] {\times} \scC$.
We obtain a $k$-simplex
\begin{equation}
	X^\dashv \circ \alpha = \alpha^* X^\dashv \colon[k] \to \Kan
\end{equation}
of $N \Kan$.
By construction, the functor $\gamma$ sends this $k$-simplex to the left fibration \smash{$r_{[k]}^* (\alpha^*X^\dashv) \longrightarrow \Delta^k$}, together with the choices of pullbacks induced by Lemma~\ref{st:r^* and change of index category}.
Therefore, we see that the morphisms $\gamma_\scC(X) \colon \Delta^n {\times} N\scC \to \scS$ and $\gamma \circ N X^\dashv \colon \Delta^n {\times} N\scC \to \scS$ agree.
\end{proof}

\begin{corollary}
\label{st:LFib classified by gamma_C(X)}
The left fibration classified by $\gamma_\scC X \colon N\scC \to \scS$ is given by $r_\scC^*X \to N\scC$.
\end{corollary}

\begin{proof}
This now follows immediately from combining Theorems~\ref{st:LFib classified by gamma o NX} and~\ref{st:gamma_C = gamma o N o dashv}.
\end{proof}

\begin{remark}
\label{rmk:functoriality of Fun(-,scS)}
For a morphism $f \colon K \to L$ of simplicial sets, we can describe the induced morphism $f^* \colon \scFun(L, \scS) \longrightarrow \scFun(K, \scS)$ as follows:
let $\{A\}$ be a consistent family of left fibrations $A(a) \to \Delta^k$, for each $k \in \NN_0$ and $a \colon \Delta^k \to \Delta^n {\times} L$, describing an $n$-simplex $F \colon \Delta^n {\times} L \to \scS$ in $\scFun(L, \scS)$ as in Remark~\ref{rmk:splcs in Fun(L, scS)}.
The consistent family $\{f^*A\}$, describing the $n$-simplex
\begin{equation}
\begin{tikzcd}[column sep=1.5cm]
	f^*F \colon \Delta^n \times K \ar[r, "1_{\Delta^n} \times f"]
	& \Delta^n \times L \ar[r, "F"]
	& \scS\,,
\end{tikzcd}
\end{equation}
assigns to a map $b \colon \Delta^k \to \Delta^n {\times} K$ the left fibration
\begin{equation}
	A \big( (1_{\Delta^n} \times f) \circ b \big) \longrightarrow \Delta^k\,,
\end{equation}
which is part of the original consistent family $\{A\}$ that describes $F$.
For each morphism $u \colon \Delta^m \to \Delta^n$ we obtain a commutative square
\begin{equation}
\begin{tikzcd}[column sep=1.5cm, row sep=1cm]
	\Delta^m \times K \ar[r, "1_{\Delta^m \times f}"] \ar[d, "u \times 1_K"']
	& \Delta^m \times L \ar[d, "u \times 1_L"]
	\\
	\Delta^n \times K \ar[r, "1_{\Delta^n \times f}"']
	& \Delta^n \times L
\end{tikzcd}
\end{equation}
This implies that $f^*$ is indeed a morphism of simplicial sets $\scFun(L, \scS) \longrightarrow \scFun(K, \scS)$.
\qen
\end{remark}

\begin{corollary}
\label{st:precomposition intertwines gammas}
Any functor $\psi \colon \scC \to \scD$ of $\rmU$-small categories induces a commutative square
\begin{equation}
\begin{tikzcd}[column sep=1.25cm]
	N \Fun(\scD, \Kan) \ar[r, "\gamma_\scD"] \ar[d, "\psi^*"']
	& \scFun(N\scD, \scS) \ar[d, "N\psi^*"]
	\\
	N \Fun(\scC, \Kan) \ar[r, "\gamma_\scC"']
	& \scFun(N\scC, \scS)
\end{tikzcd}
\end{equation}
This is compatible with composition of functors.
\end{corollary}

\begin{proof}
Given any $n$-simplex $X \colon [n] {\times} \scC \to \Kan$ in $N \Kan$, we readily observe that
\begin{equation}
	\gamma \circ N \big( X^\dashv \circ (1_{[n]} \times \psi) \big)
	= \gamma \circ N (X^\dashv) \circ (1_{\Delta^n} \times N\psi)\,.
\end{equation}
The claim now follows from Theorem~\ref{st:gamma_C = gamma o N o dashv}.
\end{proof}

\section{Localisation of simplicial diagrams}
\label{sec:gamma}

In this section we prove our main theorem:
the functor $\gamma_\scC$ exhibits $\scFun(N\scC, \scS)$ as the $\infty$-categorical localisation of $N\Fun(\scC, \Kan)$ at the objectwise weak homotopy equivalences.
To achieve this, we first show that $\gamma_\scC$ detects objectwise weak homotopy equivalences and that it is a left exact functor between $\infty$-categories with weak equivalences and fibrations (see Appendix~\ref{sec:infty-localisation} for background).

\begin{proposition}
\label{st:gamma_C preserves and detects weqs}
Suppose $f \colon X_0 \to X_1$ is a morphism in $\Fun(\scC, \Kan)$, corresponding to a functor $X_f \colon [1] {\times} \scC \to \Kan$.
Then, $f$ is an objectwise weak equivalence in $\Fun(\scC, \Kan)$ if and only if the 1-simplex $\gamma_\scC(X_f) \in \scFun(N\scC,\scS)_1$ is an equivalence in the $\infty$-category $\scFun(N\scC, \scS)$.
\end{proposition}

\begin{proof}
A morphism $F \in \scFun(N\scC,\scS)_1 \cong \rmV\sSet(\Delta^1 {\times} N\scC, \scS)$ is an objectwise equivalence precisely if, for each $c \in \scC$, the restriction \smash{$F_{|c}$} of $F$ to $\Delta^1 {\times} \{c\}$ is an equivalence in $\scS$~\cite[Cor.~3.5.12]{Cisinski:HCats_HoAlg}.

Suppose first that $f \colon X_0 \to X_1$ is a weak equivalence.
In particular, for each $c \in \scC$, the morphism \smash{$f_{|c} \colon X_0(c) \to X_1(c)$} is a weak equivalence between Kan complexes, and thus a fibrant object in $L^B_{[1]} \Fun([1], \sSet)$ in the notation of Proposition~\ref{st:r_! r^* localised}.
By Corollary~\ref{st:LFib classified by gamma_C(X)} the morphism
\begin{equation}
	\gamma_\scC (X_f) \colon \Delta^1 {\times} N\scC \longrightarrow \scS
\end{equation}
classifies the left fibration \smash{$r_{[1] {\times} \scC}^*X_f \longrightarrow \Delta^1 {\times} N\scC$}.
By Lemma~\ref{st:r^* and change of index category}, there is a cartesian square
\begin{equation}
\begin{tikzcd}
	r_{[1]}^* (X_{f_{|c}}) \ar[r] \ar[d]
	& r_{[1] \times \scC}^*(X_f) \ar[d]
	\\
	\Delta^1 \times \{c\} \ar[r]
	& \Delta^1 \times N\scC
\end{tikzcd}
\end{equation}
That is, the restriction $(\gamma_\scC X_f)_{|c} \colon \Delta^1 {\times} \{c\} \to \scS$ classifies the left fibration \smash{$r_{[1]}^* (X_{f_{|c}}) \to \Delta^1$}.
According to Corollary~\ref{st:weakly constant diags and KanFibs}, this is even a Kan fibration.
Thus, the functor $(\gamma_{[1] \times \scC} X_f)_{|c} \colon \Delta^1 \to \scS$ factors through the maximal subgroupoid $k(\scS) \subset \scS$, by~\cite[Prop.~5.2.13]{Cisinski:HCats_HoAlg}.

Now suppose that $f \colon X_0 \to X_1$ is such that $\gamma_\scC(X_f)$ is an equivalence in $\Fun(N\scC, \scS)$.
By Corollary~\ref{st:LFib classified by gamma_C(X)}, the above pullback diagram and~\cite[Prop.~5.2.13]{Cisinski:HCats_HoAlg}, it follows that
\begin{equation}
	r_{\{c\} \times [1]}^* (X_{f_{|c}}) \cong (r_{\scC \times [1]} X_f)_{|c} \longrightarrow \Delta^1
\end{equation}
is a Kan fibration, for each $c \in \scC$.
We need to show that $X_{f_{|c}}$ is local in \smash{$L^B_{[1]}\Fun([1], \sSet)$}, for each $c \in \scC$.
That will follow if \smash{$X_{f_{|c}}$} is isomorphic in $\Ho(\Fun([1], \sSet))$ to a diagram $[1] \to \sSet$ which sends the non-identity morphism in $[1]$ to a weak equivalence.
Theorem~\ref{st:h_! -| h^* QEq localised} shows that \smash{$h_{[1]}^* r_{[1]}^* X_{f_{|c}}$} is a fibrant object in \smash{$L^B_{[1]} \Fun([1], \sSet)$}, and Theorem~\ref{st:r_! r^* QEq} implies that there is an isomorphism in $\Ho(\Fun([1], \sSet))$ between $X_{f_{|c}}$ and \smash{$h_{[1]}^* r_{[1]}^* X_{f_{|c}}$}.
\end{proof}

\begin{remark}
Alternatively, one can prove Proposition~\ref{st:gamma_C preserves and detects weqs} by appealing to Corollary~\ref{st:LFib classified by gamma_C(X)}, Proposition~\ref{st:r_! r^* localised} and~\cite[Prop.~5.3.16]{Cisinski:HCats_HoAlg}.
\qen
\end{remark}

\begin{remark}
\label{rmk:N(sSet^C) as ooCat with weqs and fibs}
We turn $N \Fun(\scC, \sSet)$ into an $\infty$-category with weak equivalences and fibrations (see Definition~\ref{def:infty-cat with weqs and fibs}):
we define its class $\sfF$ of fibrations to consist of the projective fibrations and its class $\sfW$ of weak equivalences to consist of the objectwise weak equivalences.
That this satisfies the axioms in Definitions~\ref{def:class of fibs} and~\ref{def:infty-cat with weqs and fibs} follows directly from the model category axioms for the projective model structure on $\Fun(\scC, \sSet)$.
Its underlying $\infty$-category of fibrant objects (see Definition~\ref{def:ooCat of fib obs}) coincides with $N \Fun(\scC, \Kan)$ (see also Example~\ref{eg:ooCat of fib obs from MoCat}).
\qen
\end{remark}

\begin{corollary}
The functor $\gamma_\scC \colon N \Fun(\scC, \Kan) \longrightarrow \scFun(N\scC, \scS)$ factors through the $\infty$-categorical localisation $L_\sfW N \Fun(\scC, \Kan)$ and detects \emph{weak} equivalences.
\end{corollary}

\begin{remark}
\label{rmk:scS^NC as ooCat with weqs and fibs}
We turn $\scFun(N\scC, \scS)$ into an $\infty$-category with weak equivalences and fibrations by defining its class of fibrations $\sfF'$ to consist of \emph{all} morphisms and its class $\sfW'$ to consist of the honest equivalences (i.e.~the objectwise weakly invertible morphisms) in $\scFun(N\scC, \scS)$.
This satisfies the axioms in Definitions~\ref{def:class of fibs} and~\ref{def:infty-cat with weqs and fibs} because $\scS$ has finite limits.
\qen
\end{remark}

The following is a well-known result, but we state it here nevertheless; the construction of $\gamma_\scC$ allows us to see this result very explicitly.

\begin{lemma}
\label{st:gamma and holims}
Let $\scI$ be a $\rmU$-small category, and let $X \colon \scI \to \Kan$ be a functor.
Then, there is a canonical equivalence of spaces
\begin{equation}
	\gamma ( \holim_\scI X) \simeq \lim^\scS_\scI (\gamma_\scI X)\,.
\end{equation}
Here, $\gamma$ is as in Theorem~\ref{st:functor NKan -> S} and $\gamma_\scI$ is as in Theorem~\ref{st:functor NFun(C, Kan) -> Fun(NC, S)}.
\end{lemma}

\begin{proof}
The functor $\gamma_\scI X$ classifies the left fibration $r_\scI^* X \to N\scI$ (by Corollary~\ref{st:LFib classified by gamma_C(X)}).
The $\infty$-categorical limit $\lim^\scS_\scI (\gamma_\scI X)$ is computed by the space of sections of this map by~\cite[Cor.~5.4.7]{Cisinski:HCats_HoAlg} under the equivalence $\rmh \scS \simeq \Ho (\sSet)$ (see, for instance, \cite[Cor.~5.3.21, Thm.~5.4.5]{Cisinski:HCats_HoAlg} for $A=*$; note that this does not rely on the statement that $\scS$ is the $\infty$-categorical localisation of $\sSet$).
One checks that the identity map $1_{N\scI}$ coincides with $r_\scI^*(*)$ as objects of $\sSet_{/N\scI}$; here, $r_\scI^*(*)$ is the functor $r_\scI^* \colon \Fun(\scI, \sSet) \longrightarrow \sSet_{/N\scI}$ applied to the constant diagram with value $\Delta^0$.
Furthermore, by Theorem~\ref{st:r_! r^* QEq} and since each object in $\sSet_{/N\scI}$ is cofibrant, the (non-derived) counit $r_{\scI!} r_\scI^*(*) \to *$ is a cofibrant replacement of $*$ in the projective model structure on $\Fun(\scI, \sSet)$.
Consequently, we have the following canonical isomorphisms in $\Ho(\sSet)$:
\begin{align}
	\ul{\sSet_{/N\scI}} (N\scI, r_\scI^* X)
	&\cong \ul{\sSet_{/N\scI}} (r_\scI^*(*), X)
	\\
	&\cong \ul{\Fun(\scI, \sSet)} (r_{\scI!} r_\scI^*(*), X)
	\\
	&\simeq \holim_\scI(X)\,.
\end{align}
That completes the proof.
\end{proof}

\begin{proposition}
\label{st:gamma_C is Lex}
The functor $\gamma_\scC \colon N\Fun(\scC, \Kan) \longrightarrow \scFun(N\scC, \scS)$ is left exact as a functor between $\infty$-categories with weak equivalences and fibrations (see Definition~\ref{def:Lex functor}).
\end{proposition}

\begin{proof}
We need to check the axioms of Definition~\ref{def:Lex functor}.
Axiom~(1) is satisfied because $\gamma_\scC(*)$ is the functor $N \scC \to \scS$ classifies the left fibration $1_{N\scC} \colon N\scC \to N\scC$, which corresponds to the final object in $\Fun(N\scC, \scS)$.
(Alternatively, the final object of $\scFun(N\scC, \scS)$ is the essentially unique functor whose value on each $c \in \scC$ is a final object $e \in \scS$, and we have $(\gamma_\scC (*))(c) \simeq \gamma_{[0]}(*(c)) \simeq e$.)
Axiom~(2) holds true because each morphism in $\scS$ is a fibration (see Remark~\ref{rmk:scS^NC as ooCat with weqs and fibs}) and the functor $\gamma_\scC$ preserves weak equivalences (by Proposition~\ref{st:gamma_C preserves and detects weqs}).
For axiom~(3), consider morphisms $v \colon Y' \to Y$ and $f \colon X \to Y$ in $\Fun(\scC, \Kan)$, where $f$ is a projective fibration.
Let $\mathrm{CSp} = \{0 \to 1 \leftarrow 2\}$ denote the cospan category, and let $D \colon \mathrm{CSp} \times \scC \to \Kan$ denote the diagram corresponding to the cospan $Y' \to Y \leftarrow X$ in $\Fun(\scC, \sSet)$.
We have to show that the square
\begin{equation}
\begin{tikzcd}
	\gamma_\scC \big( Y' \times_Y X \big) \ar[r] \ar[d]
	& \gamma_\scC(X) \ar[d, "\gamma_\scC(f)"]
	\\
	\gamma_\scC(Y') \ar[r, "\gamma_\scC(v)"']
	& \gamma_\scC(Y)
\end{tikzcd}
\end{equation}
is cartesian in $\scFun(N\scC, \scS)$.
The cospan obtained by omitting the top-left corner of this diagram is equivalently described by a digram $\tilde{D} \colon \Lambda^2_2 \times N\scC \to \scS$ (since $N \mathrm{CSp} = \Lambda^2_2$).
Explicitly, this diagram is given as \smash{$\tilde{D} = \gamma_{\mathrm{CSp} \times \scC} D$}.

There is a canonical morphism \smash{$\gamma_\scC(Y' \times_Y X) \to \lim_{\Lambda^2_2} D$}, and it suffices to show that this morphism is an equivalence in $\Fun(N\scC, \scS)$.
However, since limits in $\scFun(N\scC, \scS)$ are computed objectwise, we may restrict ourselves to the case where $\scC = *$ is the final category.
That is, we may replace the $\infty$-category $\scFun(N\scC, \scS)$ by $\scS$.
In that case, we need to show that the canonical morphism
\begin{equation}
	\gamma(Y' \times_Y X) \longrightarrow \lim^\scS_{\Lambda^2_2} (\gamma_{\mathrm{CSp}} D)
\end{equation}
is an equivalence (with $\gamma$ as in Theorem~\ref{st:functor NKan -> S}).
We have shown in Lemma~\ref{st:gamma and holims} that there is a canonical equivalence
\begin{equation}
	\gamma \big( \holim_{\mathrm{CSp}} D \big)
	\simeq \lim^\scS_{\Lambda^2_2} (\gamma_{\mathrm{CSp}} D)\,.
\end{equation}
However, by our assumptions on the cospan $Y' \to Y \leftarrow X$, the pullback $Y' \times_Y X$ already computes the homotopy limit of $D$, showing that $Y' \times_Y X \to \lim^\scS(\gamma_{\mathrm{CSp}} D)$ is indeed an equivalence.
\end{proof}

\begin{theorem}
\label{st:localisation theorem}
Let $\scC$ be a $\rmU$-small category.
The functor $\gamma_\scC \colon N \Fun(\scC, \Kan) \longrightarrow \scFun(N\scC, \scS)$ exhibits $\scFun(N\scC, \scS)$ as the $\infty$-categorical localisation of $\Fun(\scC, \Kan)$ at the objectwise weak equivalences.
\end{theorem}

\begin{proof}
Using the notation of Remarks~\ref{rmk:N(sSet^C) as ooCat with weqs and fibs} and~\ref{rmk:scS^NC as ooCat with weqs and fibs}, we have a commutative diagram
\begin{equation}
\begin{tikzcd}[column sep=2cm, row sep=1cm]
	N\Fun(\scC,\Kan) \ar[r, "\gamma_\scC"] \ar[d, "\ell_\sfW"']
	& \scFun(N\scC, \scS) \ar[d, equal]
	\\
	L_\sfW N\Fun(\scC,\Kan) \ar[r, "\bbR \gamma_\scC"']
	& L_{\sfW'} \scFun(N\scC, \scS)
\end{tikzcd}
\end{equation}
The vertical functors are localisation functors; the equality on the right-hand side holds true because we have chosen $\sfW'$ to consist precisely of the honest equivalences in $\scFun(N\scC, \scS)$.
Our goal is to show that the right derived functor
\begin{equation}
	\bbR \gamma_\scC \colon L_\sfW N\Fun(\scC, \Kan) \longrightarrow L_{\sfW'} \scFun(N\scC, \scS) = \scFun(N\scC, \scS)
\end{equation}
is an equivalence.
By Propositions~\ref{st:gamma_C is Lex} and~\ref{st:equiv criterion for Lex functors} it suffices to show that the induced functor
\begin{equation}
	\rmh \bbR \gamma_\scC \colon \rmh \big( L_\sfW N\Fun(\scC, \Kan) \big) \longrightarrow \rmh \big( \scFun(N\scC, \scS) \big)
\end{equation}
on homotopy categories is an equivalence (of ordinary categories).
By~\cite[Rmk.~7.1.6]{Cisinski:HCats_HoAlg} there is a canonical equivalence of categories
\begin{equation}
	\rmh \big( L_\sfW N\Fun(\scC, \Kan) \big)
	\simeq \Fun(\scC, \Kan)[\sfW^{-1}]\,,
\end{equation}
where on the right-hand side we have the ordinary 1-categorical localisation.
The latter can be computed has the model-categorical homotopy category $\Ho(\Fun(\scC,\sSet))$ of the projective model structure.

We claim that there is a commutative diagram
\begin{equation}
\label{eq:loc diag for gamma_C}
\begin{tikzcd}[column sep=2cm, row sep=1cm]
	\Fun(\scC,\Kan) \ar[r, "\rmh \gamma_\scC"] \ar[d]
	& \rmh \scFun(N\scC, \scS) \ar[d, "\RepFib"]
	\\
	\Ho \big( \Fun(\scC, \sSet) \big) \ar[r, "\bbR r_\scC^*"']
	& \Ho (\sSet_{/N\scC})
\end{tikzcd}
\end{equation}
of (ordinary) categories and functors.
The functor $\RepFib$ sends an $\infty$-functor $F \colon N\scC \to \scS$ to the left fibration it classifies; we refer to~\cite[Par.~5.3.14]{Cisinski:HCats_HoAlg} and the proof of~\cite[Thm.~5.4.5]{Cisinski:HCats_HoAlg} for details.
The functor $\bbR r_\scC^*$ is an equivalence by Theorem~\ref{st:r_! r^* QEq}, as is the functor $\RepFib$~\cite[Thm.~5.4.5]{Cisinski:HCats_HoAlg} (note that Cisinski writes $\mathit{LFib}(N\scC)$ for $\Ho (\sSet_{/N\scC})$).
Granted the claim that diagram~\eqref{eq:loc diag for gamma_C} commutes, it follows that the functor $\rmh \gamma_\scC$ exhibits $\rmh \scFun(N\scC, \scS)$ as the localisation of $\Fun(\scC, \sSet)$ at the objectwise weak equivalences, as ordinary categories.
Consequently, the induced functor from the homotopy category of the localisation
\begin{equation}
	\Fun(\scC, \Kan)[\sfW^{-1}]
	\simeq \rmh \big( L_\sfW N\Fun(\scC, \Kan) \big)
	\longrightarrow \rmh \scFun(N\scC, \scS)
\end{equation}
is an equivalence.
By the universal property of 1-categorical localisation this functor is canonically equivalent to $\rmh \bbR \gamma_\scC$.
We thus infer that $\rmh \bbR \gamma_\scC$ is an equivalence as well.

It remains to show that diagram~\eqref{eq:loc diag for gamma_C} commutes.
The commutativity on objects follows directly from Corollary~\ref{st:LFib classified by gamma_C(X)}; we can write
\begin{equation}
	\RepFib \circ \gamma_\scC (X)
	= (r_\scC^* X \to N\scC)\,.
\end{equation}
To see the commutativity on morphisms, let $f \colon X_0 \to X_1$ be a morphism in $\Fun(\scC, \Kan)$.
Equivalently, this is encoded in a functor $X_f \colon [1] \times \scC \to \Kan$.
The corresponding 1-simplex in $\scFun(N\scC, \scS)$ classifies the left fibration
\begin{equation}
	r_{[1] \times \scC}^* X_f \longrightarrow \Delta^1 \times N\scC\,.
\end{equation}
The morphism in $\Ho (\sSet_{/N\scC})$ obtained from this by applying the functor $\RepFib$ is described as follows (see~\cite[Par.~5.3.14]{Cisinski:HCats_HoAlg}):
first, note that because both objects are fibrant and cofibrant the set of morphisms $\Ho (\sSet_{/N\scC})(r_\scC^*X_0, r_\scC^*X_1)$ coincides with the set $\sSet_{/N\scC}(r_\scC^*X_0, r_\scC^*X_1)/{\sim_l}$ of left homotopy equivalence classes of morphisms in $\sSet_{/N\scC}(r_\scC^*X_0, r_\scC^*X_1)$ (in the sense of, for instance,~\cite[Def.~1.2.4]{Hovey:MoCats}).
Consider the following diagram in $\sSet$:
\begin{equation}
\label{eq:lift for r_C^* and RepFib}
\begin{tikzcd}
	\Delta^{\{0\}} \times r_\scC^*X_0 \ar[d, hookrightarrow] \ar[r]
	& r_{[1] \times \scC}^* X_f \ar[d]
	\\
	\Delta^1 \times r_\scC^* X_0 \ar[r] \ar[ur, dashed, "\varphi" description]
	& \Delta^1 \times N\scC
\end{tikzcd}
\end{equation}
It admits a lift $\varphi$ because the left-hand vertical morphism is a cofinal inclusion and hence left anodyne~\cite[Cor.~4.1.9]{Cisinski:HCats_HoAlg}, and the right-hand vertical morphism is a left fibration.
Taking the fibre of $\varphi$ over $\Delta^{\{1\}} \hookrightarrow \Delta^1$ (and using Lemma~\ref{st:r^* and change of index category}), we obtain a morphism $\varphi_1 \colon r_\scC^* X_0 \longrightarrow r_\scC^* X_1$.
Its class in $\sSet_{/N\scC}(r_\scC^*X_0, r_\scC^*X_1)/{\sim_l}$ is the image of $X_f$ under $\RepFib \circ \, \rmh \gamma_\scC$ (where $\sim_l$ denotes taking left homotopy classes); see~\cite[Par.~5.3.14]{Cisinski:HCats_HoAlg} and the proof of~\cite[Thm.5.4.5]{Cisinski:HCats_HoAlg}.
This left homotopy class is independent of the choice of lift $\varphi$ by~\cite[Prop.~5.3.15]{Cisinski:HCats_HoAlg}.

We have to show that there exists a choice for $\varphi$ such that $\varphi_1 \sim_l r_\scC^*(f)$.
This is greatly simplified by the description of $r_\scC^*$ from Lemma~\ref{st:r_C^*F|n,a using bbSigma^n}:
each $n$-simplex of $\Delta^1 {\times} N\scC$ arises from a unique pair $(u, \alpha)$ of functors
\begin{equation}
	u \colon [n] \to [1]
	\qqandqq
	\alpha \colon [n] \to \scC
\end{equation}
by applying the Yoneda embedding of $\bbDelta$ to $u$ and the nerve functor to $\alpha$ (recall that we denote the map $\Delta^n \to \Delta^1$ arising from $u$ again by $u$).
We may describe each $n$-simplex in \smash{$r_{[1] \times \scC}^*X_f$} over an $n$-simplex $(u, N\alpha) \in \Delta^1_n \times N_n \scC$ as a family $z = (z_{ij})_{i,j}$, where $z_{ij} \in X_f(u(j), \alpha_j)_i$, for each $0 \leq i \leq j \leq n$, such that $z$ forms a section as in Lemma~\ref{st:r_C^*F|n,a using bbSigma^n} (compare also~\eqref{eq:Sigma alpha} and surrounding text).
That is, $z$ satisfies the identities
\begin{equation}
	X_f(u(j-1,j), \alpha_{j-1,j})(z_{i,j-1})
	= z_{ij}
	= d_{i+1}(z_{i+1,j})\,,
	\quad
	\forall\ 0 \leq i < j \leq n\,,
\end{equation}
where $(u(j-1,j), \alpha_{j-1,j})$ is the image under $(u,\alpha)$ of the unique morphism $(j{-}1) \to j$ in $[n]$.

We can naturally split this family into two parts, $z = (x,y)$, where
\begin{alignat}{3}
	& x_{ij} \in X_0(\alpha_j)_i\,,
	&& \qquad
	&& u(j) = 0\,,
	\\
	& y_{ij} \in X_1(\alpha_j)_i\,,
	&& \qquad
	&& u(j) = 1\,.
\end{alignat}
The fact that $z = (x,y)$ forms a section amounts to these data satisfying the conditions
\begin{alignat}{3}
\label{eq:identities for lift varphi}
	X_0(\alpha_{j-1,j})_i (x_{ij}) &= x_{ij} = d_{i+1} (x_{i+1,j})\,,
	&& \qquad
	&& j \in [n] \text{ with } u(j) = 0\,,
	\\*
	(f_{|\alpha_j})_i \circ X_0(\alpha_{j-1,j})_i (x_{i,j-1}) &= y_{ij} = d_{i+1} (y_{i+1,j})\,,
	&& \qquad
	&& j \in [n] \text{ with } u(j-1) = 0 \text{ and } u(j) = 1\,,
	\\*
	X_1(\alpha_{j-1,j})_i (y_{ij}) &= y_{ij} = d_{i+1} (y_{i+1,j})\,,
	&& \qquad
	&& j \in [n] \text{ with } u(j-1) = 1\,.
\end{alignat}

With these preparations, we construct a lift $\varphi$ in diagram~\eqref{eq:lift for r_C^* and RepFib} as follows:
consider an $n$-simplex $(u, N\alpha, x = (x_{ij})_{i,j})$ in $\Delta^1 \times r_\scC^*X_0$ over the $n$-simplex $(u, N\alpha) \in (\Delta^1 {\times} N\scC)_n$; that is, $x \in r_\scC^*X_0$ is an $n$-simplex over $N\alpha$.
We define $\varphi(u, N\alpha, x)$ as the pair $z_{ij} = (x'_{ij}, y_{ij})$, where we set $x'_{ij} = x_{ij} \in X_0(\alpha_j)_i$ whenever $u(j) = 0$ and $y_{ij} = f_{|\alpha_j}(x_{ij}) \in X_1(\alpha_j)_i$ whenever $u(j) = 1$.
This satisfies the identities~\eqref{eq:identities for lift varphi} because $x$ is an $n$-simplex of $r_\scC^*X_0$ and $f \colon X_0 \to X_1$ is a natural transformation of diagrams of simplicial sets, i.e.~it is compatible with the face and degeneracy maps of $X_0$ and $X_1$.
For the same reasons, it follows that this defines a morphism
\begin{equation}
	\varphi \colon \Delta^1 \times r_\scC^* X_0 \longrightarrow r_{[1] \times \scC}^*X_f
\end{equation}
as desired.
The resulting map $\varphi_1 \colon r_\scC^*X_0 \to r_\scC^*X_1$ coincides with $r_\scC^*(f)$.
\end{proof}

\begin{corollary}
\label{st:localisation thm for non-Kan}
Let $R \colon \sSet \to \Kan \subset \sSet$ be a fibrant replacement functor with natural weak equivalence $\rho \colon 1_\sSet \to R$.
Then, the composition
\begin{equation}
\begin{tikzcd}[column sep=1.5cm]
	\gamma'_\scC \colon N \Fun(\scC, \sSet) \ar[r, "{N(R \circ (-))}"]
	& N \Fun(\scC, \Kan) \ar[r, "\gamma_\scC"]
	& \scFun(N\scC, \scS)
\end{tikzcd}
\end{equation}
exhibits $\scFun(N\scC, \scS)$ as the $\infty$-categorical localisation of $\Fun(\scC, \sSet)$ at the objectwise weak equivalences.
\end{corollary}

\begin{proof}
This follows readily from~\cite[Thm.~7.5.18]{Cisinski:HCats_HoAlg} and the fact that, viewing $N \Fun(\scC, \sSet)$ as an $\infty$-category with weak equivalences and fibrations, the $\infty$-subcategory $N \Fun(\scC, \Kan) \subset N \Fun(\scC, \sSet)$ is precisely the underlying $\infty$-category of fibrant objects (see Remark~\ref{rmk:N(sSet^C) as ooCat with weqs and fibs}).
In particular, the functor $N(R \circ (-))$ induces an equivalence on the level of localisations.
\end{proof}

\begin{remark}
Recall that the nerve functor $N \colon \Cat \to \sSet$ is fully faithful as a functor of 1-categories (this holds true within any universe).
We can rephrase Theorem~\ref{st:gamma_C = gamma o N o dashv} by saying that there is a commutative square of ($\rmV$-small) simplicial sets
\begin{equation}
\label{eq:[L,Fun(NC-)] = 0}
\begin{tikzcd}[column sep=1cm, row sep=1cm]
	N\Fun(\scC, \Kan) \ar[d, "\gamma_\scC"'] \ar[r, "\cong"]
	& \scFun(N\scC, N\Kan) \ar[d, "\gamma_*"]
	\\
	\scFun(N\scC, \scS) \ar[r, equal]
	& \scFun(N\scC, \scS)
\end{tikzcd}
\end{equation}
Theorem~\ref{st:localisation theorem} implies that $\gamma_\scC$ exhibits the localisation of $\Fun(\scC, \Kan)$ at the objectwise weak equivalences.
Applying this in the case where $\scC = *$ and using Theorem~\ref{st:gamma_C = gamma o N o dashv} we also obtain that $\gamma \colon N\Kan \to \scS$ exhibits the $\infty$-categorical localisation of $\Kan$ at the weak equivalences in the Kan-Quillen model structure.
In light of this, we can now interpret diagram~\eqref{eq:[L,Fun(NC-)] = 0} as an explicit manifestation of~\cite[Prop.~7.9.2]{Cisinski:HCats_HoAlg}, i.e.~the fact that localisation commutes with forming functor $\infty$-categories.
\qen
\end{remark}

\section{Computational properties of the localisation functor}
\label{sec:gamma_C in computations}

In this section we exhibit some of the pleasant computational properties of the localisation functor $\gamma_\scC$:
we show that, for $F, G \in \Fun(\scC, \Kan)$, the $\infty$-categorical mapping space between $\gamma_\scC F$ and $\gamma_\scC G$ in $\scFun(N\scC, \scS)$ can be computed as the model-categorical mapping space.
The same holds true for the internal hom (or exponential) in $\scFun(N\scC, \scS)$.
Moreover, we show that the left and right ($\infty$-categorical) Kan extensions of $\scS$-valued functors can be computed by the left and right homotopy Kan extensions of $\sSet$-valued functors.
These insights themselves are not new (see, for instance,~\cite[Thm.~4.2.4.1]{Lurie:HTT}), but here we are able to derive them in a particularly straightforward fashion, building on~\cite[Cor.~5.4.7]{Cisinski:HCats_HoAlg}, the results of this paper so far and manipulations of model-categorical mapping spaces.
Throughout this section, let $\scC$ be a $\rmU$-small category.
Recall that we denote the simplicial hom sets in a simplicially enriched category $\scM$ by $\ul{\scM}(-,-)$.

\begin{proposition}
\label{st:mapping spaces (gamma_C F, gamma_cG)}
Let $F, G \colon \scC \to \Kan$.
There exist canonical equivalences in $\scS$,
\begin{equation}
	\scFun(N\scC, \scS)(\gamma_\scC F, \gamma_\scC G)
	\simeq \ul{\sSet_{/N\scC}}(r_\scC^*F, r_\scC^*G)
	\simeq \ul{\Fun(\scC, \sSet)} (r_{\scC!} r_\scC^* F, G)\,,
\end{equation}
and a further equivalence
\begin{equation}
	\ul{\Fun(\scC, \sSet)} (r_{\scC!} r_\scC^* F, G)
	\simeq \ul{\Fun(\scC, \sSet)} (Q F, G)
\end{equation}
where $QF$ is any functorial cofibrant replacement of $F$ in the projective model structure on $\Fun(\scC, \sSet)$.
These equivalences are natural in $F$ and $G$.
\end{proposition}

This can be derived from general arguments about localisations along the lines of~\cite[Sec.~7.10]{Cisinski:HCats_HoAlg}, but we give a more direct proof:

\begin{proof}
The first equivalence follows directly from Corollary~\ref{st:LFib classified by gamma_C(X)} and~\cite[Cor.~5.4.7]{Cisinski:HCats_HoAlg}.
The second equivalence can be obtained, for instance, by using that $r_{\scC!} \dashv r_\scC^*$ is a simplicial Quillen equivalence:
since $F$ is fibrant and $r_\scC^*F$ is cofibrant, the derived counit of this adjunction provides an equivalence $r_{\scC!} r_\scC^*F \to F$.
It thus exhibits $r_{\scC!} r_\scC^*F$ as a cofibrant replacement of $F$ in the projective model structure on $\Fun(\scC, \sSet)$.
For any other functorial cofibrant replacement $QF$, there is a natural zig-zag of weak equivalences
\begin{equation}
\begin{tikzcd}
	r_{\scC!} r_\scC^*F
	& Q(r_{\scC!} r_\scC^*F) \ar[l, "\sim"'] \ar[r, "\sim"]
	& QF
\end{tikzcd}
\end{equation}
between cofibrant objects, which the functor $\ul{\Fun(\scC, \sSet)}(-,G)$
sends to an equivalence in $\scS$.
\end{proof}

A particular choice of cofibrant replacement functor for the projective model structure on $\Fun(\scC, \sSet)$ is given by  Dugger's cofibrant replacement functor~\cite[Sec.~2.6]{Dugger:Universal_HoThys}; we denote this functor by $Q^p$ (note that since here we use simplicial pre\textit{co}sheaves on $\scC$ instead of simplicial presheaves, we need to introduce an additional $(-)^\opp$ in Dugger's conventions).
Explicitly, we have
\begin{equation}
	(Q^p F)_n(c)
	= \coprod_{c_0 \to \cdots \to c_n \to c} F_n(c_0)\,,
\end{equation}
where the coproduct is taken over chains of morphisms in $\scC$.
The functor $Q^p$ comes with a natural weak equivalence to the identity functor on $\Fun(\scC, \sSet)$, which we denote by $q^p$.

\begin{corollary}
\label{st:exponential in scS^NC}
For each $F, G \in \Fun(\scC, \Kan)$, there is an equivalence, natural in $F$ and $G$,
\begin{equation}
	(\gamma_\scC G)^{\gamma_\scC F}
	\simeq \gamma_\scC(G^{Q^p F})\,.
\end{equation}
\end{corollary}

\begin{proof}
By the Yoneda Lemma, it suffices to show that there is a natural equivalence
\begin{equation}
	\scFun(N\scC, \scS) \big( -, (\gamma_\scC G)^{\gamma_\scC F} \big)
	\simeq \scFun(N\scC, \scS) \big( -, \gamma_\scC (G^{Q^p F}) \big)
\end{equation}
of functors \smash{$\scFun(N\scC, \scS)^\opp \longrightarrow \widehat{\scS}$} (where \smash{$\widehat{\scS}$} is the $\infty$-category of $\rmV$-small spaces).
By Theorem~\ref{st:localisation theorem} $\gamma_\scC$ is a localisation functor.
It thus induces an equivalence~\cite[Def.~7.1.2]{Cisinski:HCats_HoAlg}
\begin{equation}
	\gamma_\scC^* \colon \scFun \big( \scFun(N\scC, \scS)^\opp, \widehat{\scS} \big)
	\longrightarrow \scFun_\sfW \big( N\Fun(\scC, \Kan)^\opp, \widehat{\scS} \big)\,,
\end{equation}
where on the right-hand side
\begin{equation}
	\scFun_\sfW \big( N\Fun(\scC, \Kan)^\opp, \widehat{\scS} \big)
	\subset \scFun \big( N\Fun(\scC, \Kan)^\opp, \widehat{\scS} \big)
\end{equation}
denotes the full $\infty$-subcategory on those functors which send weak equivalences to equivalences.
Therefore, it suffices to show that there is an equivalence of mapping spaces
\begin{equation}
	\scFun(N\scC, \scS) \big( \gamma_\scC A, (\gamma_\scC G)^{\gamma_\scC F} \big)
	\simeq \scFun(N\scC, \scS) \big( \gamma_\scC A, \gamma_\scC (G^{Q^p F}) \big)\,,
\end{equation}
natural in $A \in \Fun(\scC, \Kan)$.

Since $\gamma_\scC$ is left exact (Proposition~\ref{st:gamma_C is Lex}) and both $A$ and $F$ are fibrant objects in $N\Fun(\scC, \Kan)$ (seen as an $\infty$-category of fibrant objects---see Remark~\ref{rmk:N(sSet^C) as ooCat with weqs and fibs}), we have an equivalence
\begin{equation}
	\scFun(N\scC, \scS)(\gamma_\scC A, (\gamma_\scC G)^{\gamma_\scC F})
	\simeq \scFun(N\scC, \scS)(\gamma_\scC (A \times F), \gamma_\scC G)\,.
\end{equation}
By Proposition~\ref{st:mapping spaces (gamma_C F, gamma_cG)} there is a further equivalence
\begin{equation}
	\scFun(N\scC, \scS)(\gamma_\scC (A \times F), \gamma_\scC G)
	\simeq \ul{\Fun(\scC, \sSet)}\big( Q^p (A \times F), G \big)\,.
\end{equation}
For this choice of cofibrant replacement functor there exists a canonical commutative diagram
\begin{equation}
\begin{tikzcd}
	Q^p(A \times F) \ar[rr] \ar[dr, "\sim", "q^p_{A \times F}"']
	& & Q^p A \times Q^p F \ar[dl, "\sim"', "q^p_A \times q^p_F"]
	\\
	& A \times F &
\end{tikzcd}
\end{equation}
whose horizontal arrow stems from the canonical functors of slice categories
\begin{equation}
	\scC_{/A_n \times F_n} \longrightarrow \scC_{/A_n} \times \scC_{/F_n}
\end{equation}
for each $n \in \NN_0$.
The right-hand arrow is a weak equivalence since $q^p_A \colon Q^p A \to A$ and $q^p_F \colon Q^p F \to F$ are so and the injective model structure on $\Fun(\scC, \sSet)$ is cartesian monoidal~\cite[Prop.~4.51]{Barwick:Enriched_B-Loc} (it has the same weak equivalences as the projective model structure).
It follows that the top morphism is a weak equivalence as well.
Thus, we have natural equivalences
\begin{align}
	\scFun(N\scC, \scS)(\gamma_\scC A, (\gamma_\scC G)^{\gamma_\scC F})
	&\simeq \ul{\Fun(\scC, \sSet)}\big( Q^p (A \times F), G \big)
	\\
	&\simeq \ul{\Fun(\scC, \sSet)} (Q^p A, G^{Q^p F})
	\\
	&\simeq \scFun(N\scC, \scS) \big( \gamma_\scC A, \gamma_\scC (G^{Q^p F}) \big)\,,
\end{align}
where the last step is again Proposition~\ref{st:mapping spaces (gamma_C F, gamma_cG)}.
\end{proof}

Let $\scD$ be a second $\rmU$-small category.
We now compare $\infty$-categorical and homotopy Kan extensions along functors $\pi \colon \scC \to \scD$.

\begin{proposition}
\label{st:ooRan via hoRan}
Let $\pi \colon \scC \to \scD$ be a functor, and let $X \colon \scC \to \Kan$ be a diagram.
The $\infty$-categorical right Kan extension $(N\pi)_* \gamma_\scC(X)$ of $\gamma_\scC(X)$ classifies the left fibration $r_\scD^* \hoRan_\pi (X) \longrightarrow N\scD$, where $\hoRan_\pi(X)$ denotes the homotopy right Kan extension of $X$ along $\pi$.
Equivalently, there is a canonical equivalence
\begin{equation}
	(N\pi)_* \gamma_\scC(X) \simeq \gamma_\scD (\hoRan_\pi X)\,.
\end{equation}
\end{proposition}

\begin{proof}
We show that $\gamma_\scD(\hoRan_\pi X)$ represents the functor
\begin{equation}
	\scFun(N\scC, \scS)(N\pi^*(-), \gamma_\scC X)
	\colon \scFun(N\scD, \scS) \longrightarrow \widehat{\scS}\,.
\end{equation}
To that end, let $F \colon N\scD \to \scS$ be a functor.
By the same logic as in the proof of Corollary~\ref{st:exponential in scS^NC}, we may precompose the left-hand argument of $\scFun(N\scC, \scS)(-,-)$ with the localisation functor $\gamma_\scC$.
That is, we may assume that $F$ is of the form $\gamma_\scD Y$ for some functor $Y \colon \scD \to \Kan$.
Furthermore, we may assume that $X \colon \scC \to \Kan$ is fibrant in the \textit{injective} model structure on $\Fun(\scC, \sSet)$; we can always replace $X$ by an injectively fibrant approximation, up to an objectwise weak equivalence, which is preserved by both $\hoRan_\pi$ and $\gamma_\scD$ (see Proposition~\ref{st:gamma_C preserves and detects weqs}).
In particular, we then have that $\hoRan_\pi X \simeq \Ran_\pi X$, since
\begin{equation}
\begin{tikzcd}
	\pi^* : \Fun(\scD, \sSet) \ar[r, shift left=0.15cm, "\perp"' yshift=0.05cm]
	& \Fun(\scC, \sSet) : \Ran_\pi \ar[l, shift left=0.15cm]
\end{tikzcd}
\end{equation}
is a Quillen adjunction with respect to the injective model structures.

There are canonical natural equivalences of spaces
\begin{align}
	\scFun(N\scD, \scS) \big( \gamma_\scD Y, \gamma_\scD (\hoRan_\pi X) \big)
	&\overset{(1)}{\simeq} \ul{\sSet_{/N\scD}} \big( r_\scD^* Y, r_\scD^* (\Ran_\pi X) \big)
	\\
	&\overset{(2)}{\cong} \ul{\Fun(\scD, \sSet)} (r_{\scD!} r_\scD^* Y, \Ran_\pi X)
	\\
	&\overset{(3)}{\simeq} \ul{\Fun(\scD, \sSet)} (Y, \Ran_\pi X)
	\\
	&\overset{(4)}{\cong} \ul{\Fun(\scC, \sSet)} (\pi^*Y, X)
	\\
	&\overset{(5)}{\simeq} \ul{\Fun(\scC, \sSet)} (r_{\scC!} r_\scC^*\pi^*Y, X)
	\\
	&\overset{(6)}{\cong} \ul{\sSet_{/N\scC}} (r_\scC^* \pi^*Y, r_\scC^* X)
	\\
	&\overset{(7)}{\simeq} \Fun(N\scC, \scS) (\gamma_\scC \pi^*Y, \gamma_\scC X)
	\\*
	&\overset{(8)}{\simeq} \scFun(N\scC, \scS) (N\pi^* \gamma_\scC Y, \gamma_\scC X)\,.
\end{align}
Equivalence~(1) stems from Proposition~\ref{st:mapping spaces (gamma_C F, gamma_cG)}.
For step~(2), we have used that the adjunction $r_{\scC!} \dashv r_\scC^*$ is simplicial~\cite[Rmk.~4.4]{HM:Left_fibs_and_hocolims}.
Step~(3) uses that $Y$ is projectively fibrant, that $r_{\scD!} \dashv r_\scD^*$ is a Quillen equivalence between the projective and the covariant model structure~\cite[Thm.~C]{HM:Left_fibs_and_hocolims}, and that injectively fibrant objects are also projectively fibrant in $\Fun(\scD, \sSet)$.
Thus, the counit $r_{\scD!} r_\scD^* Y \to Y$ is, in particular, a weak equivalence between cofibrant objects in the injective model structure, and $\Ran_\pi X$ is injectively fibrant.
Step~(4) uses the universal property of Kan extensions.
Steps~(5), (6) and (7) use the same arguments as steps (3), (2) and (1) (in that order) with $\scC$ in place of $\scD$, and the final step~(8) is an application of Corollary~\ref{st:precomposition intertwines gammas}.
\end{proof}

\begin{proposition}
\label{st:ooLan via hoLan}
Let $\pi \colon \scC \to \scD$ and $X \colon \scC \to \Kan$ be functors.
Let $R^p_\scD$ denote a projectively fibrant replacement in $\Fun(\scD, \sSet)$
The $\infty$-categorical left Kan extension $(N\pi)_! \gamma_\scC(X)$ classifies the left fibration $r_\scD^* R^p_\scD \hoLan_\pi (X) \longrightarrow N\scD$.
Equivalently, there is a canonical equivalence
\begin{equation}
	(N\pi)_! \gamma_\scC(X) \simeq \gamma_\scD (R^p_\scD \circ \hoLan_\pi X)\,.
\end{equation}
\end{proposition}

Note that $\gamma_\scD (R^p \circ \hoLan_\pi X)$ is the image of $\hoLan_\pi X$ under the localisation functor in Corollary~\ref{st:localisation thm for non-Kan}.

\begin{proof}
As in the proof of Corollary~\ref{st:exponential in scS^NC} and Proposition~\ref{st:ooRan via hoRan}, since $\gamma_\scC$ is a localisation functor it suffices to consider mapping spaces to objects of the form $F = \gamma_\scD Y \in \scFun(N\scD, \scS)$, for $Y \in \Fun(\scD, \Kan)$.
In fact, up to weak equivalence, we may assume that $Y$ is even injectively fibrant in $\Fun(\scD, \sSet)$.
There are canonical natural equivalences (in \smash{$\widehat{\scS}$})
\begin{align}
	\scFun(N\scD, \scS)(\gamma_\scD R^p_\scD \hoLan_\pi X, \gamma_\scD Y)
	&\overset{(1)}{\simeq} \ul{\sSet_{/N\scD}} (r_\scD^* R^p_\scD \hoLan_\pi X, r_\scD^* Y)
	\\
	&\overset{(2)}{\simeq} \ul{\Fun(\scD, \sSet)} (r_{\scD!} r_\scD^* R^p_\scD \hoLan_\pi X, Y)
	\\
	&\overset{(3)}{\simeq} \ul{\Fun(\scD, \sSet)} (R^p_\scD \hoLan_\pi X, Y)
	\\*
	&\overset{(4)}{\simeq} \ul{\Fun(\scD, \sSet)} (R^p_\scD \hoLan_\pi r_{\scC!} r_\scC^*X, Y)
\end{align}
In step~(1) we have use~Proposition~\ref{st:mapping spaces (gamma_C F, gamma_cG)}.
Step~(2) uses that $r_{\scD!} \dashv r_\scD^*$ is a simplicial adjunction.
In steps~(3) and~(4) we have used that $Y$ is injectively fibrant as well as that the injective model structure on $\Fun(\scD, \sSet)$ is simplicial (as in step~(3) of the proof of Proposition~\ref{st:ooRan via hoRan}).
We further have equivalences of spaces
\begin{align}
	\ul{\Fun(\scD, \sSet)} (R^p_\scD \hoLan_\pi r_{\scC!} r_\scC^*X, Y)
	&\overset{(5)}{\simeq} \ul{\Fun(\scD, \sSet)} (\hoLan_\pi r_{\scC!} r_\scC^*X, Y)
	\\
	&\overset{(6)}{\simeq} \ul{\Fun(\scD, \sSet)} (\Lan_\pi r_{\scC!} r_\scC^*X, Y)
	\\
	&\overset{(7)}{\simeq} \ul{\Fun(\scC, \sSet)} (r_{\scC!} r_\scC^*X, \pi^*Y)
	\\
	&\overset{(8)}{\simeq} \ul{\sSet_{/N\scC}} (r_\scC^*X, r_\scC^* \pi^*Y)
	\\
	&\overset{(9)}{\simeq} \scFun(N\scC, \scS) \big( \gamma_\scC X, \gamma_\scC (\pi^* Y) \big)
	\\*
	&\overset{(10)}{\simeq} \scFun(N\scC, \scS) (\gamma_\scC X, (N\pi)^* \gamma_\scD Y)\,.
\end{align}
Step~(5) again uses the injective model structure on $\Fun(\scD, \sSet)$, the weak equivalence from the identity to $R^p$ and that $Y$ is injectively fibrant.
Step~(6) uses that $r_{\scC!} r_\scC^* X$ is projectively cofibrant, so that its homotopy left Kan extension is weakly equivalent to its strict left Kan extension.
Step~(7) is where the defining property of the 1-categorical left Kan extension enters (note that the adjunction $\Lan_\pi \dashv \pi^*$ is simplicial; this holds true since, for each $Z \in \Fun(\scC, \sSet)$, the functor $Z \times (-)$ commutes with colimits).
In step~(8) we have used the simplicial adjunction $r_{\scC!} \dashv r_\scC^*$, and step~(9) is Proposition~\ref{st:mapping spaces (gamma_C F, gamma_cG)}.
Step~(9) is an application of Proposition~\ref{st:mapping spaces (gamma_C F, gamma_cG)}.
Finally, in step~(10) we have used Corollary~\ref{st:precomposition intertwines gammas}.
\end{proof}

\begin{appendix}

\section{Background on $\infty$-categorical localisation}
\label{sec:infty-localisation}

For the reader's convenience, we recall some definitions and results on localisation of $\infty$-categories from~\cite[Ch.~7]{Cisinski:HCats_HoAlg}

\begin{definition}
\label{def:class of fibs}
\emph{\cite[Def.~7.4.6]{Cisinski:HCats_HoAlg}, \cite{Cisinski:HCats_HoAlg_Errata}}
Let $\scA$ be an $\infty$-category with a fixed final object $e$.
A \emph{class of fibrations in $\scA$} is a subobject $\sfF \subset \scA$ with the following properties:
\begin{myenumerate}
\item $\sfF$ contains all identity morphisms (i.e.~all degenerate edges).

\item $\sfF$ is closed under composition.

\item For any morphisms $g \colon y' \to y$ and $f \colon x \to y$ in $\scA$ such that $y' \to e$, $y \to e$ and $x \to e$ are in $\sfF$, the pullback $y' \times_y x$ in $\scA$ exists, and, for any cartesian square
\begin{equation}
\begin{tikzcd}
	x' \ar[r] \ar[d, "f'"']
	& x \ar[d, "f"]
	\\
	y' \ar[r, "g"']
	& y
\end{tikzcd}
\end{equation}
in $\scA$, the morphism $f'$ is in $\sfF$ whenever $f$ is so.
\end{myenumerate}
Morphisms in $\sfF$ are called \emph{fibrations}.
Given an $\infty$-category with a class of fibrations $(\scA, \sfF)$, we call an object $y \in \scA$ \emph{fibrant} if the morphism $y \to e$ is in $\sfF$.
\end{definition}

In particular, since $\sfF$ contains all identity morphisms, the final object $e \in \scA$ is necessarily fibrant.

\begin{definition}
\label{def:infty-cat with weqs and fibs}
\emph{\cite[Def.~7.4.12]{Cisinski:HCats_HoAlg}}
An \emph{$\infty$-category with weak equivalences and fibrations} is a triple $(\scA, \sfW, \sfF)$, where $\sfW \subset \scA$ is an $\infty$-subcategory (in particular closed under composition) and $\sfF \subset \scA$ is a class of fibrations, such that the following properties are satisfied:
\begin{myenumerate}
\item The class $\sfW$ satisfies the two-out-of-three property.

\item For any cartesian square
\begin{equation}
\begin{tikzcd}
	x' \ar[r] \ar[d, "f'"']
	& x \ar[d, "f"]
	\\
	y' \ar[r, "g"']
	& y
\end{tikzcd}
\end{equation}
in $\scA$, in which $f$ is a fibration between fibrant objects and $y'$ is fibrant, if $f$ is additionally in $\sfW$, then so is $f'$.

\item For each map $f \colon x \to y$ in $\scA$ with $y$ fibrant, there exists a commutative diagram (a 2-simplex)
\begin{equation}
\begin{tikzcd}[column sep=1cm]
	& x' \ar[d, "p"]
	\\
	x \ar[ur, "w"] \ar[r, "f"']
	& y
\end{tikzcd}
\end{equation}
with $w \in \sfW$ and $p \in \sfF$.
\end{myenumerate}
The morphisms in $\sfW$ are called \emph{weak equivalences}.
The morphisms in $\sfF \cap \sfW$ are called \emph{trivial fibrations}.
\end{definition}

\begin{example}
\label{rmk:MoCat M gives (NM, W, F)}
Let $\scM$ be a model category.
Let $\sfW, \sfF \subset N\scM$ denote the simplicial subsets generated by those edges which are weak equivalences or fibrations in $\scM$, respectively.
The triple $(N\scM, \sfW, \sfF)$ is an $\infty$-category with weak equivalences and fibrations.
\qen
\end{example}

\begin{remark}
Suppose that $\scA$ is an $\infty$-category such that defining $\sfF$ to consists of \textit{all} morphisms in $\scA$ makes the pair $(\scA, \sfF)$ into an $\infty$-category with fibrations.
Then $\scA$ necessarily has finite limits (since it has a final object and pullbacks~\cite[Thm.~7.3.27]{Cisinski:HCats_HoAlg}).
\qen
\end{remark}

\begin{definition}
\label{def:ooCat of fib obs}
\emph{\cite[Def.~7.5.7]{Cisinski:HCats_HoAlg}}
An $\infty$-category with weak equivalences and fibrations $(\scA, \sfW, \sfF)$ in which each object is fibrant is called an \emph{$\infty$-category of fibrant objects}.
\end{definition}

\begin{example}
\label{eg:ooCat of fib obs from MoCat}
In the setting of Remark~\ref{rmk:MoCat M gives (NM, W, F)}, let $\scM_f \subset \scM$ denote the full subcategory on the fibrant objects in $\scM$.
Let $\sfW_f \subset N\scM_f$ (resp.~$\sfF_f \subset N\scM_f$) denote the simplicial subset generated by those edges in $N\scM$ which are in $\sfW$ (resp.~in $\sfF$).
Then, the triple $(N\scM_f, \sfW_f, \sfF_f)$ is an $\infty$-category of fibrant objects.
\qen
\end{example}

\begin{definition}
\label{def:Lex functor}
\emph{\cite[Def.~7.5.2]{Cisinski:HCats_HoAlg}}
Let $(\scA, \sfW_\scA, \sfF_\scA)$ and $(\scB, \sfW_\scB, \sfF_\scB)$ be two $\infty$-categories with weak equivalences and fibrations.
A functor $F \colon \scA \to \scB$ is called \emph{left exact} if it satisfies the following properties:
\begin{myenumerate}
\item $F$ preserves final objects.

\item $F$ sends fibrations \emph{between fibrant objects} to fibrations and trivial fibrations \emph{between fibrant objects} to trivial fibrations.

\item $F$ maps cartesian squares
\begin{equation}
\begin{tikzcd}
	x' \ar[r] \ar[d, "f'"']
	& x \ar[d, "f"]
	\\
	y' \ar[r, "g"']
	& y
\end{tikzcd}
\end{equation}
in $\scA$, where $f$ is a fibration and $y$ and $y'$ are fibrant, to cartesian squares in $\scB$.
\end{myenumerate}
\end{definition}

In particular, by properties (1) and (2), $F$ preserves fibrant objects.
It also follows from an adaptation of Ken Brown's Lemma to the present $\infty$-categorical context~\cite[Prop.~7.4.13]{Cisinski:HCats_HoAlg} that any left exact functor preserves weak equivalences \textit{between fibrant objects}~\cite[Cor.~7.4.14]{Cisinski:HCats_HoAlg}.
As a consequence, one obtains an $\infty$-categorical notion of derived functors between $\infty$-categorical localisations:
first, given $\infty$-categories $\scA$ and $\scB$ with collections of weak equivalences $\sfW_\scA$ and $\sfW_\scB$, respectively, as well as a functor $F \colon \scA \to \scB$ which preserves weak equivalences, let $\overline{F} \colon L_{\sfW_\scA} \scA \longrightarrow L_{\sfW_\scB} \scB$ denote the unique functor induced on localisations.

Now, let $(\scA, \sfW_\scA, \sfF_\scA)$ be an $\infty$-category with weak equivalences and fibrations, $(\scB, \sfW_\scB)$ an $\infty$-category with an $\infty$-subcategory of weak equivalences and $F \colon \scA \to \scB$ a functor of $\infty$-categories which sends weak equivalences between fibrant objects to weak equivalences.
Let $\scA_f \subset \scA$ be the full $\infty$-subcategory on the fibrant objects; this is an $\infty$-category with weak equivalences and fibrations when endowed with the restrictions of $\sfW_\scA$ and $\sfF_\scA$.
Let $R \colon L_{\sfW_\scA} \scA \to L_{\sfW_\scA} \scA_f$ be a weak inverse to the equivalence $\overline{\iota}$ induced by the inclusion $\iota \colon \scA_f \hookrightarrow \scA$ (see~\cite[Thm.~7.5.18]{Cisinski:HCats_HoAlg}) and let $\ell_\scB \colon \scB \to L_{\sfW_\scB}\scB$ denote the localisation functor.

\begin{definition}
The right derived functor of $F$ is~\cite[Par.~7.5.25, Def.~7.5.26]{Cisinski:HCats_HoAlg}
\begin{equation}
	\bbR F = R^* \overline{\ell_\scB \circ \iota^*F} = \overline{(\ell_\scB \circ F \circ \iota)} \circ R
	\colon L_{\sfW_\scA} \scA \longrightarrow L_{\sfW_\scB} \scB\,.
\end{equation}
\end{definition}

The most important result about right derived functors in the context of this paper is

\begin{theorem}
\label{st:equiv criterion for Lex functors}
\emph{\cite[Cor.~7.6.11]{Cisinski:HCats_HoAlg}}
Let $(\scA, \sfF_\scA, \sfW_\scA)$ and $(\scB, \sfF_\scB, \sfW_\scB)$ be $\infty$-categories with fibrations and weak equivalences, and let $F \colon \scA \to \scB$ be a left exact functor.
The right derived functor $\bbR F \colon L_{\sfW_\scA} \scA \to L_{\sfW_\scB} \scB$ is an equivalence if and only if it is so on the level of homotopy categories; that is, $\rmh \bbR F \colon \rmh L_{\sfW_\scA} \scA \to \rmh L_{\sfW_\scB} \scB$ is an equivalence of \emph{ordinary} categories.
\end{theorem}

\end{appendix}

\begin{small}

\makeatletter

\interlinepenalty=10000

\makeatother

\bibliographystyle{alphaurl}
\addcontentsline{toc}{section}{References}
\bibliography{Space_loc_Bib}

\vspace{0.5cm}

\noindent
Mathematical Institute, University of Oxford.
\\
severin.bunk@maths.ox.ac.uk

\end{small}

\end{document}

%% file: Space_loc_Defs.tex
\newcommand{\im}{\operatorname{im}}
\newcommand{\RepFib}{\operatorname{RepFib}}

\newcommand{\univ}{{\mathrm{univ}}}
\newcommand{\pr}{\mathrm{pr}}

\newcommand{\rmU}{{\mathrm{U}}}

\newcommand{\rmh}{\mathrm{h}}

\newcommand{\opp}{\mathrm{op}}

\newcommand{\Lan}{\operatorname{Lan}}
\newcommand{\Ran}{\operatorname{Ran}}
\newcommand{\hoRan}{\operatorname{hoRan}}
\newcommand{\hoLan}{{\operatorname{hoLan}}}

\newcommand{\holim}{\operatorname{holim}}

\newcommand{\Ho}{{\operatorname{Ho}}}

\newcommand{\rmR}{{\mathrm{R}}}

\newcommand{\rmV}{{\mathrm{V}}}

\renewcommand{\lim}{\operatorname{lim}}

\newcommand{\scB}{\mathscr{B}}
\newcommand{\scI}{\mathscr{I}}

\newcommand{\scC}{\mathscr{C}}
\newcommand{\Kan}{{\mathscr{K}\hspace{-0.02cm}\mathrm{an}}}
\newcommand{\scD}{\mathscr{D}}
\newcommand{\scE}{\mathscr{E}}

\newcommand{\scA}{{\mathscr{A}}}
\newcommand{\scS}{\mathscr{S}}
\newcommand{\scM}{{\mathscr{M}}}

\newcommand{\sfF}{{\mathsf{F}}}

\newcommand{\sfW}{\mathsf{W}}

\newcommand{\NN}{\mathbb{N}}

\newcommand{\bbL}{\mathbb{L}}
\newcommand{\bbR}{\mathbb{R}}

\newcommand{\ul}[1]{\underline{#1}}

\newcommand{\Set}{{\mathscr{S}\mathrm{et}}}
\newcommand{\sSet}{{\mathscr{S}\mathrm{et}_{\hspace{-.03cm}\bbDelta}}}
\newcommand{\scFun}{{\mathscr{F}\hspace{-.03cm}\mathrm{un}}}
\newcommand{\Fun}{{\operatorname{Fun}}}

\newcommand{\Cat}{{\mathscr{C}\mathrm{at}}}

\newenvironment{myitemize}{\begin{itemize}[itemsep=-0.1cm, leftmargin=*, topsep=0cm]}{\end{itemize}}

\newenvironment{myenumerate}{\begin{enumerate}[topsep=-\parskip+0.1cm, itemsep=0.1cm, parsep=0cm, leftmargin=*, label=(\arabic*)]}{\end{enumerate}}

\newcommand{\qen}{\hfill$\triangleleft$}

\newcommand{\qandq}{\quad \text{and} \quad}
\newcommand{\qqandqq}{\qquad \text{and} \qquad}

\newcommand{\bbDelta}{{\mathbbe{\Delta}}}
\newcommand{\bbSigma}{{\mathbbe{\Sigma}}}
\newcommand{\bbLambda}{{\mathbbe{\Lambda}}}

\newcommand{\textint}{\textstyle{\int}}